\documentclass[10pt,a4paper]{amsart}
\usepackage[latin1]{inputenc}
\usepackage{amsmath}
\usepackage{amsfonts}
\usepackage[T1]{fontenc}
\usepackage{amssymb}
\usepackage{graphicx}
\usepackage{amscd}
\usepackage{mathrsfs}
\usepackage[all]{xy}
\usepackage{overpic}
\usepackage{stackrel}

\usepackage{slashed} 

\usepackage[T1]{fontenc}
\newtheorem{theorem}{Theorem}
\usepackage{amsmath}

\newtheorem{thm}{Theorem}[section]
\newtheorem{lemma}[thm]{Lemma}
\newtheorem{prop}[thm]{Proposition}

\theoremstyle{definition}
\newtheorem{defn}[thm]{Definition}

\theoremstyle{remark}
\newtheorem{remark}[thm]{Remark}
\numberwithin{equation}{section}

\newcommand{\ad}{\text{ad}}

\newcommand{\tr}{\text{tr}}                        

\newcommand{\Om}{\Omega}

\newcommand{\Z}{\mathbb{Z}}       
\newcommand{\R}{\mathbb{R}}      

\newcommand{\RR}{{\mathbb R}}
\newcommand{\N}{{\mathbb N}}

\newcommand{\surj}{\to\kern-1.8ex\to}

\newcommand{\cA}{\mathcal{A}}

\newcommand{\cG}{\mathcal{G}}

\def\om{\omega}
\def\Om{\Omega}

\def\cA{\mathcal{A}}

\def\cG{\mathcal{G}}

\def\bfL{\mathbf{L}}

\def\bfP{\mathbf{P}}

\begin{document}

\title[$G_2$-Strominger system, $T$-duals and infinitesimal moduli]{$T$-dual solutions and infinitesimal moduli of the $G_2$-Strominger system}

\author[A. Clarke]{Andrew Clarke}
\address{Instituto de Matem\'atica, Universidade Federal do Rio de Janeiro,
Av. Athos da Silveira Ramos 149,
Rio de Janeiro, RJ, 21941-909,
Brazil.}
\email{andrew@im.ufrj.br}

\author[M. Garcia-Fernandez]{Mario Garcia-Fernandez}
\address{Instituto de Ciencias Matem\'aticas (CSIC-UAM-UC3M-UCM)\\
  Nicol\'as Cabrera 13--15, Cantoblanco\\ 28049 Madrid, Spain}
  \email{mario.garcia@icmat.es}
\author[C. Tipler]{Carl Tipler}
\address{Univ Brest, UMR CNRS 6205, Laboratoire de Math\'ematiques de Bretagne Atlantique, France
}
\email{carl.tipler@univ-brest.fr}

\date{\today}


\begin{abstract}
We consider $G_2$-structures with torsion coupled with $G_2$-instantons, on a compact $7$-dimensional manifold. The coupling is via an equation for $4$-forms which appears in supergravity and generalized geometry, known as \emph{the Bianchi identity}. First studied by Friedrich and Ivanov, the resulting system of partial differential equations describes compactifications of the heterotic string to $3$ dimensions, and is often referred to as the \emph{$G_2$-Strominger system}.
We study the moduli space of solutions and prove that the space of infinitesimal deformations, modulo automorphisms, is finite dimensional. We also provide a new family of solutions to this system, on $T^3$-bundles over $K3$ surfaces and for infinitely many different instanton bundles, adapting a construction of Fu-Yau and the second named author. In particular, we exhibit the first examples of $T$-dual solutions for this system of equations.
\end{abstract}

 \maketitle
\section{Introduction}
\label{sec:intro}
A
fundamental problem in differential geometry is the generalization of gauge theory to higher dimensional varieties.
Since the principal bundle formulation of Yang-Mills theory in the 1970s, there has been a substantial interaction between various areas of physics and differential geometry, via gauge theory. Indeed, one aim in modern mathematical gauge theory is to obtain results on the geometry and topology of higher dimensional manifolds using ideas that originate in physics.
 As initiated by Donaldson and Thomas, and Tian \cite{DT,Ti00}, these approaches require one to consider manifolds endowed with specific geometric structures, such as metrics with holonomy $SU(n)$ or $G_2$. The study of 
 gauge theory in higher dimensions has in recent years seen major developments; see for example \cite{Ha12,LoOl,SEthesis,Ta12,Wa}, to say nothing of the enormous literature on gauge theory in complex geometry. Moreover, gauge theoretic conditions can also be considered on spaces that admit certain geometric structures, but whose Riemannian holonomy group is not reduced (see \cite{BaOl,HN} and references below). 

The problem that we consider here is inspired from high-energy physics, and runs parallel to recent developments on the
Hull-Strominger system of partial differential equations in dimension $6$ \cite{HullTurin,Strom}. 
The mathematical study of the Hull-Strominger system
(see \cite{Fei,GF2,PPZ2} for recent surveys covering this topic) was initiated by Li and Yau as a natural generalization of the Calabi problem, 
and it is motivated by `Reid's fantasy' on the moduli space of complex $3$-folds with trivial canonical bundle and varying topology. In the light of \cite{AcharyaGukov}, it is conceivable that Li-Yau's proposal for the geometrization of conifold transitions and flops between K\"ahler and non-K\"ahler Calabi-Yau three-folds can be carried over into the $7$-dimensional case for the geometrization of $G_2$-{\sl transitions} \cite{CHNP}. Motivated by this, here we consider $G_2$-structures  with torsion coupled with $G_2$-instantons, by means of an equation for $4$-forms which
arises from the Green-Schwarz anomaly cancellation mechanism in string theory.
The resulting system of equations can be regarded as an analogue of the Hull-Strominger system in $7$-dimensions and was first studied
by Friedrich and Ivanov \cite{FrIv02,FrIv03}. 
Following \cite{FIUVa7}, we settle for referring to the  $G_2$-{\sl Strominger system} (in the more recent physics literature, it goes under the name of the {\sl heterotic $G_2$ system} \cite{DLMS19}).

From the point of view of physics, the $G_2$-Strominger system is a particular instance of a more general system of equations, known as {\sl the Killing spinor equations in (heterotic) supergravity}. The {\sl compactification} of the physical theory leads to the study  of models of the form $N^k\times M^{10-k}$, where $N^k$ is a $k$-dimensional Lorentzian manifold  and $M^{10-k}$ is a Riemannian spin manifold which encodes the extra dimensions of a supersymmetric vacuum. With a natural compactification ansatz, the Killing spinor equations, for a Riemannian metric $g$, a spinor $\Psi$, a function $f$ (the dilaton), a $3$-form $H$ (the NS-flux), and a connection $A$ with curvature $F_A$ on a principal $K$-bundle $P_K$ over  $M^{10-k}$, can be written as
\begin{equation}
 \label{eq:Killing spinors 10dim}
 \nabla \Psi=0, \qquad \qquad (df - \frac{1}{4}H)\cdot \Psi=0, \qquad \qquad F_A\cdot \Psi =0,
\end{equation}
where $\nabla$ is a $g$-compatible connection with skew-symmetric torsion $H$. 
Solutions to (\ref{eq:Killing spinors 10dim}) provide rich geometrical structures on $M$. If the
torsion $H$ vanishes, the existence of a parallel spinor reduces the holonomy of the Levi-Civita connection on $M$ to $SU(n), Sp(n),G_2$ or $Spin(7)$ according to its dimension.
However, the torsion-free condition, often equivalent to the condition $dH=0$ --- the so-called strong solutions ---, is very restrictive, as many interesting solutions to the equations arise in  
manifolds equiped 
with metric connections with holonomy contained in $SU(n), Sp(n),G_2$ or $Spin(7)$ but non-vanishing skew-symmetric torsion.

An interesting relaxation of the notion of strong solution is provided by the \emph{Bianchi identity} (related to the anomaly cancellation condition in string theory), which requires a correction of $dH$ of the form
\begin{equation}
 \label{eq:Anomaly}
 dH=\frac{\alpha}{4}( \tr(F_A\wedge F_A)-\tr(R_\nabla\wedge R_\nabla))
\end{equation}
where $\alpha$ is a positive constant, $F_A$ is as in \eqref{eq:Killing spinors 10dim}, and $R_\nabla$ is the curvature of an additional linear connection $\nabla$ on the tangent bundle of $M$. 
The extra requirements for a solution of the Killing spinor equations \eqref{eq:Killing spinors 10dim} and the Bianchi identity \eqref{eq:Anomaly} to provide with a supersymmetric vacuum of the theory is given by the instanton condition  \cite{Ivan09}
\begin{equation}
 \label{eq:instanton}
R_\nabla \cdot \Psi =0.
\end{equation}
In a $6$-dimensional compact manifold $M$, the combination of the above mentioned equations (\ref{eq:Killing spinors 10dim}), (\ref{eq:Anomaly}) and (\ref{eq:instanton}) leads to the Hull-Strominger system. In this paper, we provide new solutions and initiate the mathematical study of the moduli space of solutions to the system obtained by coupling Equations (\ref{eq:Killing spinors 10dim}), (\ref{eq:Anomaly}) and (\ref{eq:instanton}) in $7$ dimensions -- the $G_2$-Strominger system -- that we introduce next.

Consider $M^7$ a compact oriented smooth manifold. Then,
the equations (\ref{eq:Killing spinors 10dim}), (\ref{eq:Anomaly}) and (\ref{eq:instanton})
are equivalent to the following system \cite{FrIv03}:
\begin{equation}  \label{eq:systemG2Killing}
  \begin{split}
    d\phi\wedge\phi & =0, \ \ \ \ \ \ \ \ \ \ \ \ \ d * \phi = -4 d f \wedge *\phi ,\\
    F_A \wedge *\phi & = 0, \ \ \ \ \ \ \ \ \ \ \ \  R_\nabla \wedge *\phi  = 0,\\
      dH  & = \frac{\alpha}{4}(\tr(F_A\wedge F_A)-\tr(R_\nabla\wedge R_\nabla)),
  \end{split}
\end{equation}
where $\phi$ is a positive $3$-form that defines a $G_2$ structure on $M$, $-4df$ is the Lee form $\theta_\phi$ of $\phi$, and $H$ is the torsion of the $G_2$-structure, given by
$$
H = - *( d\phi - \theta_\phi \wedge \phi ).
$$
The first line of equations in \eqref{eq:systemG2Killing} characterizes a special type of $G_2$-structures that are conformally equivalent to coclosed $G_2$-structures of type $W_3$ \cite[Theorem 2]{FrIv03}, according to the classification by Fern\'andez and Gray \cite{FerGr}.
Some Riemannian properties of these structures are studied in \cite{FrIv03}.
The second line of equations in \eqref{eq:systemG2Killing} is the $G_2$-instanton condition, and has been the subject of important recent progress (see e.g. \cite{DS,HN,O,SW}
and the references therein). The last line, the Bianchi identity, is a defining equation for a Courant algebroid, 
and leads to a new mathematical approach to equations from string theories and supergravity theories using methods from generalized geometry \cite{Hit1} (see e.g. \cite{CMTW,GF14,grt}, in the context of heterotic supergravity).  
It should be mentioned that equations \eqref{eq:systemG2Killing} enforce $N=\mathbb{R}^3$ in the compactification. A different compactification ansatz, with $N$ anti-de Sitter space-time, leads to a more general class of solutions with $d\phi \wedge \phi = \lambda \phi \wedge * \phi$, for $\lambda \in \RR$ \cite{OssaLaSv15}.

Our study of the $G_2$-Strominger system starts with a result concerning the moduli space of solutions of \eqref{eq:systemG2Killing}. This moduli space has been widely studied in the physics literature, mainly due to the work of De la Ossa and collaborators \cite{DLMS19,OssaLaSv16,DLS18b,DLS18a,FQS}. We hope that our development here provides mathematical underpinnings for this interesting physical advances. To state our main theorem concerning this moduli space, we introduce some notation. Let $P_M$ be the bundle of oriented frames over $M$.
The group $\cG$, given as an extension of the group of diffeomorphisms isotopic to the identity by the group of gauge transformations of $P_M\times_M P_K$
acts naturally on the set of parameters $(\phi,f,\nabla,A)$ for the system (\ref{eq:systemG2Killing}), preserving solutions, and thus defining a natural set
$$
\mathcal{M} = \{(\phi,f,\nabla,A)\;  \mathrm{satisfying} \eqref{eq:systemG2Killing}\}/\cG.
$$
In Section \ref{sec:InfMod}
we use elliptic operator theory to prove that the (expected) tangent space of $\mathcal{M}$ at a solution $(\phi,f,\nabla,A)$ is finite dimensional. More precisely, we construct a finite-dimensional space of infinitesimal deformations of a solution $(\phi,f,\nabla,A)$ of \eqref{eq:systemG2Killing}, modulo the action of $\cG$. 
\begin{theorem}
\label{theointro:moduli}
 Let $M$ be a $7$-dimensional compact spin manifold and $P_K$ a principal $K$-bundle over $M$. Then the space of infinitesimal deformations of a solution to the system of equations (\ref{eq:systemG2Killing}) on $(M,P_K)$, modulo the infinitesimal $\cG$-action, is finite-dimensional.
\end{theorem}

This result can be regarded as a first step towards the construction of a natural structure of smooth manifold on $\mathcal{M}$, and shall be compared with the alternative approach taken in \cite{OssaLaSv16,DLS18b}.


To the present day, there is a handful of compact examples where our Theorem \ref{theointro:moduli} applies. Basic compact solutions to the $G_2$-Strominger system \eqref{eq:systemG2Killing} are provided by torsion-free $G_2$-structures. For this, one sets $K = G_2$ and $P_K$ the bundle of orthogonal frames of a $G_2$-holonomy metric, and defines $\nabla = A$ equal to the Levi-Civita  connection. The first compact solutions with non-zero torsion (and constant dilaton function $f$) to the $G_2$-Strominger system \eqref{eq:systemG2Killing} have been constructed in \cite{FIUV7}. Non-compact solutions to \eqref{eq:systemG2Killing} have been constructed in \cite{FIUVa7,GuNi}. In this paper, following a method initiated by Fu-Yau \cite{FY08} and used by the second author \cite{GF19b} for the $6$-dimensional Hull-Strominger system, we provide new compact examples of solutions to \eqref{eq:systemG2Killing} on torus bundles over $K3$ surfaces with associative $T^3$-fibres. More precisely, let $S$ be a $K3$ surface and let $\beta_1, \beta_2$ and $\beta_3$ be closed anti-self-dual $2$-forms on $S$ with integral cohomology classes. Each of these forms defines a circle bundle over $S$, and we consider $M$ to be the fibre product of these three circle bundles. The principal bundle $P_K$ will be the pull-back of a principal bundle on $S$. By this construction, we show that the set of data together satisfy the $G_2$-Strominger system if and only if a certain scalar function $h\in C^\infty(S)$ satisfies
\begin{eqnarray*}
\Delta h= t^2\left(|\beta_1|^2+|\beta_2|^2+|\beta_3|^2\right)+*_4\langle F_\theta\wedge F_\theta\rangle
\end{eqnarray*}
where $\langle F_\theta\wedge F_\theta\rangle$ is the quadratic curvature expression coming from the right-hand side of the Bianchi identity, and that depends on the parameter $\alpha$. As described in Section \ref{sec:ansatzconstruction}, the solutions will also depend on a parameter $t>0$ related to the size of the fibers of the torus fibration.
We denote the intersection form on second cohomology of the $K3$ surface $S$ by
$$
Q:  H^2(S,\Z) \times H^2(S,\Z)  \to  \Z.
$$
\begin{theorem}
\label{theointro:examplesnonconstant dilaton}
For any choice of $t >0$, $\alpha\in\mathbb{R}^*$ and $r\in\mathbb{N}^*$, such that
 \begin{equation}
  \label{eq:integralratiosintro}
 \frac{2t^2}{\alpha}\sum_{j=1}^3 Q\left(\left[\frac{1}{2\pi}\beta_j\right]\right)\in\mathbb{Z}
 \end{equation}
 and 
 \begin{equation}
 \label{eq:bound on rankintro}
 r\leq 24 + \frac{2t^2}{\alpha}\sum_{j=1}^3 Q\left(\left[\frac{1}{2\pi}\beta_j\right]\right),
 \end{equation}
there exist a solution of the system (\ref{eq:systemG2Killing}) on the $7$-manifold $M$ constructed as above.
\end{theorem}

We refer to Theorem \ref{theo:Ansatz} in Section \ref{sec:ansatzconstruction} for a more precise description of the solutions.  We note that this scheme of construction was already suggested in \cite[Section 6]{FIUV7}, but our solutions are genuinely different. To illustrate this, observe that, for different values of the parameters $t$, $\alpha$ and $r$, we obtain an infinite family of solutions for infinitely many different instanton bundles (see Remark \ref{rem:infinite}). 
Furthermore, we expect
that our construction provides examples of compact solutions with non-constant dilaton.



Our last result concerning the system of equations (\ref{eq:systemG2Killing}), in Section \ref{sec:Tdual}, is an explicit construction of $T$-duality for pairs of solutions of the $G_2$-Strominger system built on the associative $T^3$-fibrations over $K3$ surfaces constructed in Theorem \ref{theointro:examplesnonconstant dilaton}. This result is motivated by a recent proposal in the physics literature to extend the so-called $(0,2)$-{\sl mirror symmetry} (see e.g. \cite{MelSS}, and references therein \cite{GF19b}) to the case of seven dimensional manifolds \cite{FQS}. This new form of mirror symmetry is expected to have very different features to the more familiar mirror symmetry on manifolds of exceptional holonomy arising from type IIA/IIB string theory \cite{Acharya,Leung}, mainly due to the absence of $D$-branes in the context of the heterotic string.

To state our result, we consider as before a $K3$ surface with three closed anti-self-dual $2$-forms $\beta_1,\beta_2$ and $\beta_3$ such that $[\beta_i]\in2\pi H^2(S,\mathbb{Z})$. Suppose also that for $t>0$, $[t^2\beta_i]\in2\pi H^2(S,\mathbb{Z})$. Let $M$ be the $T^3$-bundle over $S$ determined by the triple $(\beta_1,\beta_2,\beta_3)$ and let $M^\prime$ be the bundle determined by triple $(-t^2\beta_1,-t^2\beta_2,-t^2\beta_3)$. Let $P$ and $P^\prime$ be principal $K$-bundles over $M$ and $M^\prime$ obtained by pulling back the same principal bundle $P_S$ on $S$.


Then we have (see Theorem \ref{theo:Tduals} in Section \ref{sec:examplesTduals} for a more precise statement):
 


\begin{theorem}
 \label{theointro:Tduals}
 Suppose that the triple $(\beta_i)$, together with the size $t$ satisfy the constraints \eqref{eq:integralratiosintro}, \eqref{eq:bound on rankintro}, and 
$$
[t^2\beta_i]\in2\pi H^2(S,\mathbb{Z}).
$$ 
Then, $(M,P)$ and $(M^\prime,P^\prime)$ both admit solutions to the $G_2$-Strominger system and furthermore, these solutions are exchanged under $T$-duality. 
\end{theorem}
 

The proof of Theorem \ref{theointro:Tduals} builds on a general result in previous work by the second named author \cite[Theorem 7.6]{GF19b} where it was proved that the solutions of \eqref{eq:Killing spinors 10dim} and \eqref{eq:Anomaly} with the instanton ansatz \eqref{eq:instanton} for the connection $\nabla$ are exchanged by heterotic $T$-duality (in arbitrary dimensions). This notion of $T$-duality adapted to the equations of the heterotic string was introduced by Baraglia and Hekmati in \cite{BarHek}, building on \cite{BEM,CaGu}. To our knowledge, Theorem \ref{theointro:Tduals} provides the first examples of $T$-dual solutions of the $G_2$-Strominger system in the literature. Following \cite{FQS} we speculate that our $T$-dual solutions correspond to seven dimensional $(0,2)$-{\sl mirrors}. Dual $T^3$ fibrations over $K3$ surfaces have been considered before in the context of the heterotic string via a complicated chain of string dualities \cite{HMS}. It would be interesting to explore the relation between these pairs of heterotic string backgrounds and our $T$-dual solutions in Theorem \ref{theointro:Tduals}. 


\textbf{Acknowledgments:} The authors would like to thank Gueo Grantcharov for suggesting the torus invariant solutions in Theorem \ref{theo:Ansatz}, Xenia de la Ossa for helpful conversations, and the anonymous referee for providing many helpful suggestions to improve a former version of this paper. CT is partially supported by ANR project EMARKS No ANR-14-CE25-0010 and by CNRS grant PEPS jeune chercheur. AC would like to acknowledge the financial support of the CNRS and CAPES-COFECUB that made possible his visit to LMBA-UBO. 
MGF was partially supported by a Marie Sklodowska-Curie grant (MSCA-IF-2014-EF-655162), from the European Union's Horizon 2020 research and innovation programme, and by the Spanish MINECO under ICMAT Severo Ochoa project No. SEV-2015-0554, and under grant No. MTM2016-81048-P.


\section{Background on $G_2$-structures and the $G_2$-Strominger system}
\label{sec:background}
In this section we introduce the necessary material on $G_2$-structures and the $G_2$-Strominger system. 
Let $M$ be a $7$-dimensional compact spin manifold. We will denote by $\Om^\bullet(M)$, or $\Om^\bullet$, the space of  differential $\bullet$-forms on $M$.

\subsection{$G_2$-structures and instantons}
\label{sec:backgroundG2}
A $G_2$-structure on $M$ is given by a $3$-form $\phi$ such that each point of $M$, there exists a basis $\{\varepsilon^i\}$ of $T^*M$ such that $\phi$ is given by 
\begin{eqnarray*}
\phi=\varepsilon^{123} -\varepsilon^1\wedge (\varepsilon^{45}+\varepsilon^{67})
-\varepsilon^2\wedge (\varepsilon^{46}+\varepsilon^{75})
-\varepsilon^3\wedge (\varepsilon^{47}+\varepsilon^{56})
\end{eqnarray*}
where $\varepsilon^{ij}=\varepsilon^i\wedge\varepsilon^j$, etc. The exceptional compact simple Lie group $G_2$ is isomorphic to the stabilizer of the corresponding $3$-form on $\mathbb{R}^7$, under the action of $GL(7,\mathbb{R})$. The form $\phi$ algebraically determines a (positive definite)  Riemannian metric  $g_\phi$ on $M$ with respect to which the coframe $\{\varepsilon^i\}$ is orthonormal. We take $M$ to be oriented by the volume form $\varepsilon^{1234567}$. We will denote by $\Om^3_+(M)$ the space of such positive $3$-forms $\phi$.

Let $\phi_0$ be the standard flat $G_2$-structure on $\mathbb{R}^7$. As representations of $G_2$, $\Lambda^2\mathbb{R}^7$ and $\Lambda^3\mathbb{R}^7$ decompose into irreducible subspaces. In particular, $\Lambda^2\mathbb{R}^7=\Lambda^2_7\oplus\Lambda^2_{14}$ and $\Lambda^3\mathbb{R}^7=\Lambda^3_1\oplus\Lambda^3_7\oplus\Lambda^3_{27}$, where $k$ is the dimension of the component $\Lambda^i_k$. These subspaces  can be understood explicitly. The space $\Lambda^2_7$ is the set of elements $*(\alpha\wedge*\phi_0)$ for $\alpha\in\Lambda^1\cong\mathbb{R}^7$, with $\Lambda^3_7$ defined similarly. The space $\Lambda^2_{14}\subseteq\Lambda^2$ corresponds to the Lie sub-algebra $\mathfrak{g}_2\subseteq \mathfrak{so}(7)$ and is the kernel of the map $*\phi_0\wedge\cdot:\Lambda^2\to \Lambda^6$. $\Lambda^3_1$ is spanned by $\phi_0$. The final space $\Lambda^3_{27}$ can be identified with the space of trace-free symmetric bilinear forms on $\mathbb{R}^7$, though we will not need this characterisation. Note also that the sets  of $4$ and $5$-forms decompose according to $\Lambda^i_k=*(\Lambda^{7-i}_k)$.
As a consequence, on any $7$-manifold equipped with a $G_2$-structure $\phi$, the bundles of $2$ and $3$-forms similarly decompose into direct sums of subbundles. We denote by $\Omega^i_k$ the space of $i$-forms that lie in the subbundle $\Lambda^i_k$. 

The different possible algebraic classes of $G_2$-structures on $7$-manifolds have been classified by Fern\'andez and Gray \cite{FerGr}, according to the irreducible $G_2$-representation spaces in which the covariant derivative $\nabla^\phi\phi$ takes its values. That is,
\begin{eqnarray*}
\nabla^\phi\phi \in W_1 \oplus W_2 \oplus W_3 \oplus W_4\subseteq \Lambda^1\otimes\Lambda^3.
\end{eqnarray*}
For example, the $G_2$-structure is {\sl torsion-free} if the components in all four subspaces vanish. The $G_2$-structure is {\sl nearly-parallel} if only the component in $W_1$ is non-zero. Moreover, the components in this decomposition can be determined from the {\sl type}-decomposition of the exterior derivatives of $\phi$ and $*\phi$. That is, there exist $\tau_1\in\Omega^0$, $\tau_2\in\Omega^2_{14}$, $\tau_3\in\Omega^3_{27}$ and $\tau_4\in\Omega^1$ such that
\begin{eqnarray*}
d\phi &=& \tau_1*\phi+3\tau_4\wedge \phi+*\tau_3,\\
d*\phi &=& 4\tau_4\wedge *\phi +*\tau_2,
\end{eqnarray*}
and such that $\tau_k$ vanishes if and only if the component of $\nabla^\phi\phi$ in $W_k$ vanishes. 
 In this paper, we will be interested in $G_2$-structures defined by $3$-forms that satisfy 
\begin{eqnarray*}
d\phi\wedge \phi=0, &\ \ \ \ \ &  d(*\phi)= -4df\wedge*\phi 
\end{eqnarray*}
for $f$ a smooth real valued function.  That is, $\tau_1=\tau_2=0$ and $\tau_4=-df$. In particular, the conformally equivalent $G_2$-structure $\phi^\prime=e^{3f}\phi$ satisfies $d(*^\prime\phi^\prime)=0$ so our equations are for a $G_2$-structure to be conformally equivalent to one purely of type $W_3$.

In addition to the above two equations, we study $G_2$-structures that also satisfy a third condition that couples $(\phi,f)$ to the curvature of a connection on an auxiliary principal bundle on $M$. For any connection $A$ on $P$, the curvature takes values in the bundle $\Lambda^2\otimes\ad P$. We say that $A$ is a $G_2$-{\sl instanton} if the curvature takes values in the subbundle $\Lambda^2_{14}\otimes\ad P$ associated to the Lie subalgebra $\mathfrak{g}_2\subseteq\mathfrak{so}(7)$. As noted above, this is equivalent to the condition $F_A\wedge *\phi=0$. We note here that, in contrast to the nearly-parallel case (another case in which $\phi$ is coclosed, see \cite{BaOl}), a $G_2$-instanton with respect to a $W_3$-type $G_2$-structure does not necessarily satisfy the Yang-Mills equations. 

We now take this opportunity to explicitly define a set of first order differential operators originally studied by Bryant and Harvey \cite{Bryant}. We set $\Omega_1=\Omega^0$, $\Omega_7=\Omega^1$, $\Omega_{14}=\Omega_{14}^2$ and $\Omega_{27}=\Omega^3_{27}$. Then, for each $i,j\in\{1,7,14,27\}$, we have a first order differential operator $d_j^i:\Omega_i\to\Omega_j$, defined by the exterior derivative composed with projection onto the appropriate subspace. These operators are studied in detail in \cite[Section 5.2]{Bryant} and used in Section \ref{sec:InfMod}. To aid the exposition in that section, we define explicitly here those maps that appear later. For $f\in\Omega_1$, $\alpha\in\Omega_7$, $\beta\in\Omega_{14}$ and $\gamma\in\Omega_{27}$, we have
\begin{eqnarray*}
&d^1_7f= df,     &\\
 d^7_1\alpha=d^*\alpha=-*d*\alpha, & d_7^7\alpha=*(\*\phi\wedge d\alpha),& d^7_{27}\alpha=*\pi_{27}(d*(\alpha\wedge\phi)),\\
d^7_{14}\alpha=\pi_{14}(d\alpha), & d^{14}_7\beta =-*d*\beta=d^*\beta=-*d(\phi\wedge\beta),   & d^{14}_{27}\beta=\pi_{27}(d\beta),\\
 & d_7^{27}\gamma=*(\phi\wedge*d\gamma),& d_{14}^{27}\gamma =-\pi_{14}(*d*\gamma).
\end{eqnarray*}
These formulae are aided by the fact that the projection $\pi_{14}:\Lambda^2\to \Lambda^2_{14}$ is given by $\pi_{14}=2/3\mathrm{Id}-1/3*(\phi\wedge\cdot)$. The projection $\pi_{27}$ can also be calculated.

\subsection{The $G_2$-Strominger system}
\label{sec:system}
Let $P$ be a principal $G$-bundle over $M$ for a given Lie group $G$. We assume that there is a non-degenerate bi-invariant pairing on the Lie algebra $\mathfrak{g}$ of $G$:
$$
\langle\,,\rangle \colon \mathfrak{g} \otimes \mathfrak{g} \to \mathbb{R}.
$$
We are interested in the \emph{$G_2$-Strominger system}:
\begin{equation}
\label{eq:G2 system}
  \begin{split}
    d\phi\wedge\phi & =0, \ \ \ \ \ \ \ \ \ \ \ \ \ \ d * \phi = -4 d f \wedge *\phi ,\\
    -d(*( d\phi + 4 df \wedge \phi ))  & = \langle F_\theta\wedge F_\theta\rangle,\\
     F_\theta \wedge *\phi & = 0,
  \end{split}
\end{equation}
where the $3$-form $\phi\in \Om^3_+(M)$ defines a $G_2$ structure, $f\in C^\infty(M)$, $\theta$ is a connection in $P$, and $F_\theta$ denotes the curvature of $\theta$. Note that the Hodge dual $*$ is taken with respect to the metric given by the $G_2$-structure, inducing some non-linearity in the system. 

By the first line in \eqref{eq:G2 system}, $-4df$ is the Lee form of $\phi$ (see e.g. \cite[Proposition 1]{Bryant}). We will thus sometimes refer to a solution of \eqref{eq:G2 system} by a pair $(\phi,\theta)$. Moreover the form $H$ defined by 
\begin{equation}
 \label{eq:torsionform}
H=-*( d\phi + 4 df \wedge \phi )
\end{equation}
is the torsion $3$-form of the $G_2$ structure. Then, the Bianchi identity, last equation in \eqref{eq:systemG2Killing}, imposes the vanishing of the first Pontryagin class of $(P,\langle\,,\rangle)$:
$$
p_1(P)= 0.
$$
In the case of interest in Section \ref{sec:ansatzconstruction},
we will consider the system \eqref{eq:G2 system} on $E\to M$, where $E$ is an associated vector bundle
$$
E:=P\times_\rho E_0
$$
for a representation $\rho : G \to GL(E_0)$.

\begin{remark}
 \label{rem:physics system}
In physics literature, the principal bundle $P$ is taken to be the fibre product $P_M\times_M P_K$ of the principal bundle of orthogonal frames $P_M$ of $(M,g)$ by a principal $K$-bundle $P_K$ over $M$, with a compact group $K$. The pairing $\langle\,,\rangle$ is taken to be of the form
\begin{equation}
\label{eq:pairingc}
  \langle\,,\rangle = \frac{\alpha}{4}(\tr_\mathfrak{k} - \tr_{\mathfrak{so}}),
\end{equation}
for a positive constant $\alpha$, and where $- \tr_\mathfrak{k}$ denotes the Killing form on $\mathfrak{k}$ while $- \tr_{\mathfrak{so}}$ denotes the Killing form on $\mathfrak{so}(7,\R)$.
 In this situation, the topological constraint for the Bianchi identity is 
$$
p_1(P_M) = p_1(P_K).
$$
An additional condition is that $\theta$ is a product connection $\theta = \nabla \times A$, with $\nabla$ a metric connection on $TM$. With these conditions, we recover the system \eqref{eq:systemG2Killing}.
\end{remark}


\section{Infinitesimal Moduli of the $G_2$-Strominger system}
\label{sec:InfMod}
In this section we consider the question of moduli and deformation of solutions of the $G_2$-Strominger system. For equations of this type, the ultimate desired result would be to show that solutions, modulo some obvious equivalence, appear in smooth families whose dimension can be calculated by the index of a certain elliptic differential operator. This is the case in classical $4$-dimensional Yang-Mills theory, after choosing a generic Riemannian metric, and in $G_2$-geometry on compact manifolds (see for example \cite{FrUhl,joy}). While such a theorem currently appears out of reach for the $G_2$-Strominger system, we can study the infinitesimal problem.

The moduli space of solutions of the $G_2$-Strominger system is the quotient space $\mathcal{M}=\mathcal{E}^{-1}(0)/\mathcal{G}$, where $\mathcal{E}$ is the non-linear operator defining the system and $\mathcal{G}$ is the symmetry group  of the system. The infinitesimal model for this space, at a point $x=(\phi,f,\theta)\in\mathcal{E}^{-1}(0)$, is the quotient vector space
\begin{eqnarray*}
H^1_x=\ker d\mathcal{E}_{x}/\mathrm{Im}\:\bfP_x
\end{eqnarray*}
where $\bfP_x$ is the operator giving the infinitesimal symmetries through $x$. In the absence of showing that $\mathcal{M}$ has the structure of a smooth manifold, we show that $H^1_x$ is finite dimensional.

\subsection{$G_2$-structures conformally of type $W_3$ and with closed torsion} 
Let $M$ be a $7$-dimensional compact spin manifold. In this section we study $G_2$-structures that are conformally of pure type $W_3$ and for which the torsion $3$-form $H$ is closed. That is, we consider the $G_2$-Strominger system in the case that the structure group of the auxiliary bundle is trivial $G=\{1\}$:
\begin{eqnarray}
 \label{eq:systemnobundle}
    d\phi\wedge\phi & =&0,\nonumber\\
    d * \phi  +4 df \wedge *\phi&=&0 , \\
    dH=-d(*( d\phi + 4 df \wedge \phi ))  & = &0.\nonumber
\end{eqnarray}
This is a simplified version of the full $G_2$-Strominger system, however we will be able to derive conclusions about infinitesimal deformations of the general system from information about this set of equations. Our first conclusion is that solutions to this system are torsion-free.
\begin{prop}\label{prop:zeroflux}
A pair $(\phi,f)$ is a solution of \eqref{eq:systemnobundle} on a compact $7$-manifold $M$ if and only if $f$ is constant and $\phi$ is torsion-free, that is, $d \phi = 0$ and $d^*\phi = 0$.
\end{prop}

This fact is well-known in the physics literature (see e.g. \cite{GaMaWa}). We give a short proof based on two methods for calculating the scalar curvature of a solution of the system \eqref{eq:systemnobundle}, one coming from the relation between Killing spinors in $7$ dimensions and conformally coclosed $G_2$-structures, and the other specifically considering the equations of motion in heterotic string theory implied by \eqref{eq:systemnobundle} (see \cite{Ivan09}).

\begin{proof}  We combine two equations that have appeared in the literature relating the solution $(\phi,f)$ to the induced Riemannian structure. Let $g=g_\phi$ be the metric determined by $\phi$. From \cite[Thm.~1.1]{Ivan09} we have, for $\theta=-4df$,
\begin{eqnarray*}
\operatorname{Ric}^g_{ij} &=&\frac{1}{4}H_{imn}H^{mn}_j +4 \nabla_i\nabla_jf,\\
\text{hence, }\ \ \ S^g&=&\frac{1}{4}\|H\|^2-4\Delta f
\end{eqnarray*}
where $\Delta =\delta d$ is the Laplacian with positive spectrum. A complementary expression is given in \cite[Eq.~1.5]{FrIv03}, without the assumption that $dH=0$, 
\begin{eqnarray*}
S^g = 16|df|^2-\frac{1}{12}\|H\|^2-12\Delta f.
\end{eqnarray*}
These can be combined to give 
\begin{eqnarray*}
16|df|^2-\frac{1}{3}\|H\|^2-16\Delta f &=&0,\\
-16e^{-f}\Delta(e^f)-\frac{1}{3}\|H\|^2 &=&0,
\end{eqnarray*}
which gives $\int e^f\|H\|^2\: \mathrm{dvol_g}=0$ and hence $H=-*(d\phi+4df\wedge\phi) =0$. This, together with the structure equation of Fern\'andez and Gray, which in this case takes the form $d(*\phi)=-4df\wedge*\phi$, implies that $df=0$ as desired. 
\end{proof}
We summarize in the next lemma various useful identities relating the operators and projections defined in Section \ref{sec:backgroundG2}.
\begin{lemma} \label{lemma:DiffIdents} Let $\phi$ be a torsion-free $G_2$-structure and let $J$ be the endomorphism
\begin{equation}
\label{eq:definitionJ}
\begin{array}{cccc}
J:&\Omega^3&\to &\Omega^3\\
 &\xi & \mapsto & \frac{4}{3}\pi_1(\xi)+\pi_7(\xi)-\pi_{27}(\xi).
\end{array}
\end{equation}
Then for any $\beta_7=*(\alpha\wedge*\phi)\in\Omega^2_7$ and $\beta_{14}\in\Omega^2_{14}$ we have
\begin{enumerate}
\item $d*Jd\beta_7=0$,
\item $\pi_7(d*Jd\beta_{14})=0$,
\item $\pi_{14}(d*Jd\beta_{14})=\Delta\beta_{14}-\pi_{14}(dd^*\beta_{14})$.
\end{enumerate}
\end{lemma}
\begin{proof}
A direct calculation gives
\begin{eqnarray*}
d*Jd\beta_7 &=&\left(\frac{-4}{7}d_7^1d_1^7\alpha -\frac{1}{3}(d_7^7)^2\alpha +\frac{1}{3}d_7^{27}d_{27}^7\alpha\right)\wedge *\phi +*\left(\frac{1}{2}d_{14}^7d_7^7\alpha +d_{14}^{27}d_{27}^7\alpha\right)
\end{eqnarray*}
which vanishes by \cite[Prop.~3]{Bryant}. The other relations are similar. In particular, from \cite{Bryant} we have $\pi_{14}(d*Jd\beta_{14})=\Delta\beta_{14}-d^7_{14}d^{14}_7\beta_{14}$ from which we obtain (3).
\end{proof}
As a consequence of this lemma we can conclude that for {\sl any} $G_2$-structure $\phi$ defining $*$ and $J$, and for $v\in T^*M$ and $\beta_{7}\in\Lambda^2_{7}$ we must have $v\wedge(*J(v\wedge\beta_7))=0$, with similar vanishing relations for the symbols of the other differential operators considered in Lemma \ref{lemma:DiffIdents}. 
We note here that the operator $J$ is given as the linear term of the map $\phi\mapsto  *\phi$ in Joyce \cite[Eq.~10.9]{joy}. The proof of this fact appears in \cite[Lemma~20]{Hit0}.

Next, we consider the deformation problem for solutions of \eqref{eq:systemnobundle}, and characterize the space of infinitesimal deformations of this system. By Proposition \ref{prop:zeroflux}, we recover with different methods an infinitesimal version of the theorem of  Joyce on the moduli of torsion-free $G_2$-structures \cite{joy,Hit0}. 
 
We take as parameter space for the deformation problem the space $  \mathcal{P}_M= \Omega^3_{+}\times {C}^\infty(M)$, with $T_{(\phi,f_0)}\mathcal{P}_M=\Omega^3(M)\times C^\infty(M)$, and suppose that $(\phi,f_0)$ is a solution to \eqref{eq:systemnobundle}. Let $\mathcal{R}_M=\Omega^7\times\Omega^5\times \Omega^4$. The group $\mathrm{Diff}_0(M)$ of diffeomorphisms isotopic to the identity acts by pull-back on $\mathcal{P}_M$. 
 The linearization of this action, at $(\phi,f_0)$, is the map
\begin{eqnarray*}
\bfP_M=\bfP_{M,(\phi,f_0)}:\Gamma(TM)&\to&\Omega^3\times C^\infty(M),\\
V&\mapsto & (\mathcal{L}_V\phi,\mathcal{L}_Vf_0)=(di_V\phi,0).
\end{eqnarray*}
We consider the linearization of the non-linear operator defining the left-hand side of Equations \eqref{eq:systemnobundle}. This gives $\bfL_M:T_{(\phi,f_0)}\mathcal{P}_M\to \mathcal{R}_M$ defined by
\begin{eqnarray}
 \label{eq: linearisation no bundle}
 \bfL_M : 
        (\dot \phi,\dot f) & \mapsto & 
        \begin{cases}
        d\dot \phi \wedge \phi, \\
        d * J \dot  \phi+4 d\dot f \wedge *\phi,  \\
       -d(*( d\dot\phi + 4d\dot  f \wedge \phi )).
\end{cases}
\end{eqnarray}
\begin{prop}
\label{prop: W3 G2 metric moduli}
Let $\phi$ be a torsion-free $G_2$-structure and $f_0$ a real constant. Then,
\begin{equation}
 \label{eq:tangent H3 R}
 \frac{\ker \bfL_M}{\mathrm{Im}\: \bfP_M}\simeq \mathcal{H}^3(M,\mathbb{R})\times \mathbb{R}
\end{equation}
 where $\mathcal{H}^3(M,\mathbb{R})$ is the space of harmonic $3$-forms on $(M,g_\phi)$.
\end{prop}
\begin{proof}
Supposing that $\bfL_M(\dot\phi,\dot f)=0$, Equation \eqref{eq: linearisation no bundle} gives that
\begin{eqnarray}
d^*d(\dot{\phi} +4\dot{f}\phi) &=&0,\nonumber \\
\label{eq:thmCal}d^*(J\dot{\phi}+4\dot{f}\phi) &=& 0.
\end{eqnarray}
In particular, $\dot{\phi}+4\dot{f}\phi$ is closed so by the Hodge theorem, $\dot{\phi}+4\dot{f}\phi =h+d\beta$, for $h$ harmonic and $\beta=\beta_7+\beta_{14}\in\Omega^2$. We claim that the component $d\beta_{14}$ must vanish. Equation \eqref{eq:thmCal} then implies that 
$
d*Jd\beta_7+d*Jd\beta_{14}-4/3d\dot{f}\wedge *\phi=0.
$
However, by Lemma \ref{lemma:DiffIdents},  $d*Jd\beta_7=0$ and $\pi_7(d*Jd\beta_{14})=0$, hence $
 \pi_{14}(d*Jd\beta_{14})-\frac{4}{3} d\dot  f \wedge *\phi=0
 $. By a consideration of type, this implies that  $\pi_{14}(d*Jd\beta_{14})=0$ and $d\dot f=0$, so $\dot f$ is constant.  
  We observe at this point that this implies that $d\dot{\phi}=0$, and so $\dot{\phi}$ automatically satisfies the first equation $d\dot{\phi}\wedge\phi=0.$
  Next, by Lemma \ref{lemma:DiffIdents}  
\begin{eqnarray}
\label{eq:LapIdentity}
0 &=& \Delta \beta_{14} -\pi_{14}(dd^*\beta_{14})
= \Delta\beta_{14} -\frac{2}{3}dd^*\beta_{14}-\frac{1}{3}d^**(\phi\wedge d^*\beta_{14}).
\end{eqnarray}
 Comparing exact and co-exact terms in this expression gives that $d^*\beta_{14}=0$ which, from the same equation, gives that $d\beta_{14}=0$. Therefore,
 \begin{eqnarray*}
\dot{\phi}+4\dot{f}\phi=h+d\beta_7=h+d\iota_V\phi,\ \ \ V\in\Gamma(TM).
\end{eqnarray*} 
Thus, we can define a map 
\begin{equation}
\begin{array}{ccc}
\ker \bfL_M  & \rightarrow & \mathcal{H}^3\times \mathbb{R}\\
(\dot\phi,\dot f) & \mapsto & (h-4\dot f \phi, \dot f).
 \end{array}
\end{equation}
This map is well defined, surjective, and has kernel the image of $\bfP_M$, thus proving \eqref{eq:tangent H3 R}.
\end{proof}

\subsection{Infinitesimal deformations of the $G_2$-Strominger system} \label{subsec: proof main theo} 
Consider now a Lie group $G$ with non-degenerate bi-invariant pairing $c$ on its Lie algebra. We let $P$ be a principal $G$-bundle over $M$. The $G_2$-Strominger system is given by the system of equations  
\begin{eqnarray}
\label{eq:DefStrom}
\mathcal{E}(x) &=&0,\\
\text{ where }\ \ \ \ \ \mathcal{E}:\Omega_+^3\times C^\infty(M)\times\mathcal{A}_P& \longrightarrow & \Omega^7\times\Omega^5\times\Omega^4\times\Omega^6(\ad P),\nonumber\\
\text{is given by }\ \ \ \ \ \mathcal{E}(\phi,f,\theta)&=&
\begin{cases}
d\phi\wedge\phi,\\
d*\phi+4df\wedge *\phi,\\
-d*(d\phi +4df\wedge\phi)-\langle F_\theta\wedge F_\theta\rangle,\\
F_\theta\wedge*\phi.
\end{cases}\nonumber
\end{eqnarray}
Here 
$\mathcal{A}_P$ is the space of connections on the principal $G$-bundle $P$ over $M$. 
We let $\mathcal{P}=\Omega_+^3(M)\times C^\infty(M)\times \mathcal{A}_P$ and $\mathcal{R}=\Omega^7\times\Omega^5\times\Omega^4\times\Omega^6(\ad P)$.  Let ${\mathcal{G}}$ be the group of diffeomorphisms of $P$ that project to define diffeomorphisms of $M$ isotopic to the identity, and that commute with the right action of $G$ on $P$. That is, ${\mathcal{G}}$ is an extension of 
$\mathrm{Diff}_0(M)$ by the group of gauge transformations $\cG_P$ of $P$, and we have the sequence
\begin{eqnarray*}
1\to \cG_P\longrightarrow {\mathcal{G}}\longrightarrow \mathrm{Diff}_0(M)\longrightarrow 1.
\end{eqnarray*} 
 The group ${\mathcal{G}}$ acts from the right 	on $\mathcal{P}$ by pull-back of forms on $M$ and pull-back of connection forms on $P$. This action preserves the set of solutions of \eqref{eq:DefStrom}.

We suppose that $x=(\phi,f,\theta)\in\mathcal{P}$ satisfies Equation \eqref{eq:DefStrom}. The infinitesimal action $\mathbf{P}=\mathbf{P}_x$ of ${\mathcal{G}}$ at $x$, and the linearization  $\mathbf{L}=\mathbf{L}_x$ of $\mathcal{E}$ at $x$ are given by
\begin{eqnarray}
 \bfP :  \Omega^0(TM)\times \Omega^0(\ad P) & \longrightarrow & \Omega^3\times C^\infty(M) \times \Omega^1(\ad P),\nonumber\\
\label{eq:inf action}      
         (V, r) & \longmapsto & (\mathcal{L}_V\phi, \mathcal{L}_Vf, d^\theta r + \iota_V F_\theta), \\
&& \nonumber\\          
          \bfL :    \Omega^3\times C^\infty(M) \times \Omega^1(\ad P) & \longrightarrow &\Omega^7 \times \Omega^5 \times \Omega^4 \times \Omega^6(\ad P),  \nonumber\\
          (\dot \phi,\dot f, \dot \theta) & \longmapsto & 
        \begin{cases}
        \bfL_1=d\dot \phi \wedge \phi + d\phi\wedge \dot \phi, \\
        \bfL_2= d * J \dot  \phi+4 d\dot f \wedge *\phi +4 df \wedge *J \dot \phi, \\
        \bfL_3=-d(*( d\dot\phi + 4d\dot f \wedge \phi )) -d(\dot *( d\phi + 4 df \wedge \phi ))\\
        \ \ \ \ \ \  -d(*( 4 df \wedge \dot \phi )) -2d\langle \dot\theta, F_\theta\rangle,\\
         \bfL_4= d^\theta\dot \theta \wedge *\phi + F_\theta \wedge *J \dot \phi,
\end{cases}
\end{eqnarray}
where for $l=3,4$, $J:\Omega^l\to \Omega^l$ is defined by formula \eqref{eq:definitionJ}.
These operators fit into the deformation complex 
\begin{eqnarray}
\label{eq:fullcomplex}
\mathrm{Lie}({\mathcal{G}})\stackrel{\bfP}{\longrightarrow} T_x\mathcal{P}\stackrel{\bfL}{\longrightarrow}\mathcal{R}.
\end{eqnarray}
The main result of this section is the ellipticity of the operator $\bfL^*\bfL+\bfP\bfP^*$, which implies :
\begin{theorem}\label{thm:InfDef_G2S}
The space $\ker\bfL/\mathrm{Im}\:\bfP$ of infinitesimal deformations of the $G_2$-Strominger system at $x$ is finite dimensional.
\end{theorem}
To prove this result we use the theory of multi-degree elliptic linear differential operators, as defined by Douglis and Nirenberg \cite{DN}. In particular, to detect ellipticity it is sufficient to consider only the highest order operators in each of the terms $\bfL_1,\ldots,\bfL_4$. Thus, the symbols of $\bfL$ and $\bfP$ are the same as the symbols of $\bfL_h$ and $\bfP_h$ defined by
\begin{eqnarray} 
 \bfP_h :  \Omega^0(T)\times \Omega^0(\ad P) & \longrightarrow & \Omega^3\times C^\infty(M) \times \Omega^1(\ad P),\nonumber\\
\label{eq:inf action high}          
          (V, r) & \longmapsto & (d\iota_V\phi, 0, d^\theta r),  \\
 \bfL_h :    \Omega^3\times C^\infty(M) \times \Omega^1(\ad P) & \longrightarrow &\Omega^7 \times \Omega^5 \times \Omega^4 \times \Omega^6(\ad P),  \nonumber\\
 \label{eq: full linearisation high}  
         (\dot \phi,\dot f, \dot \theta) & \mapsto & 
        \begin{cases}
        d\dot \phi \wedge \phi ,\\
        d * J \dot  \phi+4 d\dot f \wedge *\phi,  \\
       -d(*( d\dot\phi + 4d\dot f\wedge \phi )), \\
        d^\theta\dot \theta \wedge *\phi.
\end{cases}
\end{eqnarray}
With this simplification, the first thing to note is that the fourth equation in $\bfL_h$ is now completely decoupled from the first three. That is, the deformation complex associated to the operators $\bfP_h$ and $\bfL_h$  decomposes into the two 
\begin{eqnarray}
\label{eq:man_complex}
&\mathrm{Lie}(\mathrm{Diff}_0)\stackrel{\bfP_M}{\longrightarrow} T_{(\phi,f)}\mathcal{P}_M\stackrel{\bfL_M}{\longrightarrow}\mathcal{R}_M,&\\
\label{eq:inst_complex}
&\mathrm{Lie}(\cG_P)\stackrel{\bfP_P}{\longrightarrow} T_\theta\mathcal{A}_P\stackrel{\bfL_P}{\longrightarrow}\mathcal{R}_P.&
\end{eqnarray}
Here, as should be clear, $\mathcal{P}_M=\Omega_+^3(M)\times C^\infty(M)$, $\mathcal{R}_M=\Omega^7\times\Omega^5\times\Omega^4$ and $\mathcal{R}_P=\Omega^6(\ad P)$. Ellipticity of the operator $\bfL_h^*\bfL_h+\bfP_h\bfP_h^*$, and thus of $\bfL^*\bfL+\bfP\bfP^*$, will then follow from the ellipticity of the operators $\bfL_M^*\bfL_M+\bfP_M\bfP_M^*$ and $\bfL_P^*\bfL_P+\bfP_P\bfP_P^*$.
We consider these cases separately. 
\begin{prop}
Let $(\phi,f)\in\mathcal{P}_M$. Then, the complex \eqref{eq:man_complex} is elliptic at $T_{(\phi,f)}\mathcal{P}_M$. 
\end{prop} 
Denoting by $\sigma_{\mathbf{A},v}$ the principal symbol of a differential operator $\mathbf{A}$ at $v\in T_p^*M$,
this is to say that $\ker\sigma_{\bfL_M,v}=\mathrm{Im}\:\sigma_{\bfP_M,v}$, for any $v\in T_p^*M$. This in turn implies that the operator $\bfL_M^*\bfL_M+\bfP_M\bfP_M^*$ is elliptic in the sense of Douglis and Nirenberg.
\begin{proof}
Let $v\in T^*_pM$ be a non-zero cotangent vector on $M$ and suppose that $\sigma_{\bfL_M,v}(\dot\phi,\dot f)=(0,0,0)$. We wish to show that $\dot f=0$ and  $\dot\phi=v\wedge\iota_V\phi$ for some $V\in T_pM$. From the equation $v\wedge*(v\wedge(\dot\phi+4\dot f\phi))=0$ we deduce that $v\wedge(\dot\phi+4\dot f\phi)=0$ and $\dot\phi+4\dot f\phi=v\wedge(\beta_7+\beta_{14})$ for some $\beta=\beta_7+\beta_{14}\in\Lambda^2$. We aim to show that $v\wedge\beta_{14}=0$ and $\dot f=0$. The above, together with the equation $v\wedge (*J\dot\phi+4\dot f\phi)=0$, gives
\begin{eqnarray}
\label{eq:someequation involving symbols and betas}
v\wedge*J(v\wedge\beta_7)+v\wedge *J(v\wedge\beta_{14})=\frac{4}{3}\dot fv\wedge*\phi.
\end{eqnarray}
From the discussion after Lemma \ref{lemma:DiffIdents}, we have $v\wedge*J(v\wedge\beta_7)=0$ and $\pi_7(v\wedge*J(v\wedge\beta_{14}))=0$.
 Thus, \eqref{eq:someequation involving symbols and betas} becomes 
\begin{eqnarray*}
\pi_{14}(v\wedge *(J(v\wedge \beta_{14})))=\frac{4}{3}\dot f v\wedge * \phi.
\end{eqnarray*}
The two sides must then vanish for reasons of type and hence $\dot f=0$ and $v\wedge *J(v\wedge\beta_{14})=0$. As a consequence of the same lemma, $\pi_{14}(v\wedge*J(v\wedge\beta_{14}))$ is given by 
\begin{eqnarray*}
0&=&\pi_{14}(v\wedge *J(v\wedge\beta_{14}))\\
&=&v\wedge(\iota_{v^\#}\beta_{14})+\iota_{v^\#}(v\wedge\beta_{14})-\pi_{14}(v\wedge*(v\wedge\beta_{14}\wedge\phi)),\\
&=&v\wedge(\iota_{v^\#}\beta_{14})+\iota_{v^\#}(v\wedge\beta_{14}) -\frac{2}{3}v\wedge(\iota_{v^\#}\beta_{14})-\frac{1}{3}\iota_{v^\#}*(\phi\wedge \iota_{v^\#}\beta_{14}).
\end{eqnarray*}
This holds on $\mathbb{R}^7$, as a consequence of Equation \eqref{eq:LapIdentity}, and hence for any $G_2$-structure. Thus, comparing terms of the form $v\wedge A$ and of the form $\iota_{v^\#}B$, we conclude that $\iota_{v^\#}\beta_{14}=0$, which then implies that $\iota_{v^\#}(v\wedge\beta_{14})=0$ and hence $v\wedge\beta_{14}=0$ as desired.
\end{proof}
The second complex \eqref{eq:inst_complex} corresponds to a system parametrizing $G_2$-instantons on $P$, modulo gauge transformation, for an arbitrary $G_2$-structure. This fits into the {\sl elliptic} complex
\begin{eqnarray*}
0{\longrightarrow} \Omega^0(\mathrm{ad}(P))\stackrel{d^\theta}{\longrightarrow}
\begin{array}{c}
\Omega^1(\mathrm{ad}(P))\\
\oplus\\
\Omega^0(\mathrm{ad}(P))
\end{array}
\stackrel[*d^\theta]{*\phi\wedge d^\theta}{\longrightarrow }\Omega^6(\mathrm{ad}(P))\longrightarrow 0.
\end{eqnarray*}
In particular, ellipticity of the complex \eqref{eq:inst_complex} at the term $T_{\theta}\cA_P$ is proven in \cite[Prop.~1.22]{SEthesis}. The deformation theory of $G_2$-instantons on compact manifolds, for torsion-free $G_2$-structures, is discussed in detail in \cite{SEthesis} and \cite{Wa}.

As a consequence of the above calculations, we obtain
\begin{prop}
\label{prop:ellipticLP}
 The operator $\bfL^*\bfL+\bfP\bfP^*$ is elliptic.
\end{prop}
 The cohomology group $\ker\bfL/\mathrm{Im}\:\bfP$ is isomorphic to the kernel of the elliptic operator in Proposition \ref{prop:ellipticLP} and is hence finite dimensional. This concludes the proof of Theorem \ref{thm:InfDef_G2S}.
\begin{remark}
The authors are unaware of index theorems for mixed-degree operators of this type, however this would be the natural avenue to explore to calculate the dimension of this vector space.
\end{remark}

\begin{remark}
Similarly as in \cite{grt}, \eqref{eq:fullcomplex} can be modified to build a complex for infinitesimal deformations of the $G_2$-Strominger system with fixed string class (see Definition \ref{def:stringclass}) using generalized geometry. This other complex is for differential operators on degree $1$, and  also has finite-dimensional cohomology. We expect that this alternative approach should play an important role in future studies of the $G_2$-Strominger system in relation to mirror symmetry (see Section \ref{sec:Tdual}).
\end{remark}

\section{New solutions to the $G_2$-Strominger system}
\label{sec:ansatzconstruction}

The first solutions to the Hull-Strominger system on non-K\"ahler complex three-folds were constructed by Fu and Yau in the fundamental paper \cite{FY08}. These solutions require that the connection $\nabla$ that appears in the anomaly cancellation term is the Chern connection of the solution metric. With the different hypothesis that $\nabla$ is an instanton with respect to the solution metric, the second named author produced new solutions to the Hull-Strominger system on the same $6$-dimensional manifolds \cite{GF19b}. In this section, we show that this method can be carried over to the $7$-dimensional case, producing a new family of solutions to the system \eqref{eq:G2 system}. This was already suggested in \cite[Section 6]{FIUV7}, where the ansatz for the $G_2$-structure was considered, but without the extra data of the instantons.

There
are very few constructions of solutions to the system (\ref{eq:G2 system}). The first examples of compact solutions to this system were constructed in \cite{FIUV7} on nil-manifolds. These solutions arise in finite dimensional families.

\subsection{An ansatz on $T^3$-fibrations over hyperk\"ahler $4$-folds}
Let $(S,g)$ be a compact $4$-dimensional hyperk\"ahler manifold, with hyperk\"ahler triple of $2$-forms $\omega_1,\omega_2,\omega_3$, each of pointwise length $\sqrt{2}$. These forms are each self-dual with respect to $g$, and in fact span the set of closed self-dual $2$-forms on $S$. We consider closed anti-self-dual $2$-forms $\beta_1$, $\beta_2$, $\beta_3$ such that $\frac{1}{2\pi}\beta_i$ represent integral cohomology classes. Note that in particular $\beta_i\wedge\om_j=0$ for all $i,j$, and the forms $\beta_i$ arise as curvature forms for connections on $S^1$-bundles over $S$. We denote by $M$ the fibre product of the three circle bundles. The manifold $M$ is a compact $7$-manifold, that fibres as a principal $T^3$-bundle over $S$. 

Let $\pi:M\to S$ be the projection map, and let $\sigma=(\sigma_1,\sigma_2,\sigma_3)$ be the $T^3$-connection form on $M$, with values in $\mathbb{R}^3$, that satisfies $d\sigma=(\pi^*\beta_1,\pi^*\beta_2,\pi^*\beta_3)$.  
Let $u\in C^\infty(M,\mathbb{R})$ be a smooth real-valued function on $M$, and let $t>0$ be constant. We consider the $3$-form $\phi$ on $M$ :
\begin{eqnarray}
\label{eq:phiut}
\phi=\phi_{u,t}=t^3\sigma_1\wedge \sigma_2\wedge \sigma_3 - te^u\left(\sigma_1\wedge \omega_1+\sigma_2\wedge \omega_2+\sigma_3\wedge \omega_3\right).
\end{eqnarray}
For any function $u$ and any $t>0$, $\phi$ defines a $G_2$-structure on $M$. The induced metric $g_\phi$ and volume form $\mathrm{d}vol_\phi$ are given by
\begin{eqnarray}\label{eq:metricG2}
g_\phi&=&t^2\sum_{i=1}^3\sigma_i^2+e^u\pi^*g_S,\\
\mathrm{d}vol_\phi &=& t^3e^{2u}\sigma_{123}\wedge\pi^*\mathrm{d}vol_S,
\end{eqnarray}
where $g_S$ and $\mathrm{d}vol_S$ are respectively the hyperk\"ahler metric and volume form on $S$ associated to the triple $\{\omega_i\}$. Here and in the following, for brevity, we use the convention that $\sigma_{ij}=\sigma_i\wedge\sigma_j$, etc.
We claim that for suitable choices of $f$ and $(E,\theta)$, there exists a smooth function $u$ and positive value $t$ such that the system \eqref{eq:G2 system} is satisfied on $M$. 

\begin{prop}
\label{prop:firstequation}
The $3$ form $\phi$ satisfies
$d\phi\wedge \phi =0$.
\end{prop}
\begin{proof}
We have 
\begin{eqnarray*}
d\phi=t^3\left( \beta_1\wedge \sigma_{23}+\beta_2\wedge\sigma_{31}+\beta_3\wedge\sigma_{12}\right) 
-te^u du\wedge \sum_i\sigma_i\wedge\omega_i.
\end{eqnarray*}
Therefore, $d\phi\wedge \phi=0$ since we have
$\sigma_{ijkl}=0,$ $\omega_i\wedge\beta_j=0$ and $ du\wedge \omega_i^2=0$.
\end{proof}
\begin{lemma}
 \label{lem:Hodgedual4form}
The Hodge dual $4$-form $*\phi$ is given by
\begin{eqnarray}
\label{eq:computing the dual of phi}
*\phi=e^{2u}\frac{\omega_1^2}{2} -t^2e^u\left(\sigma_{23}\wedge \omega_1+\sigma_{31}\wedge \omega_2+\sigma_{12}\wedge\omega_3\right).
\end{eqnarray}
\end{lemma}
From this we deduce:
\begin{prop}
\label{prop:secondequation}
 The differential of $*\phi$ is given by
\begin{eqnarray*}
d(*\phi)=-t^2e^udu\wedge \left(\sigma_{23}\wedge \omega_1+\sigma_{31}\wedge \omega_2+\sigma_{12}\wedge\omega_3\right).
\end{eqnarray*}
Thus, the second part of \eqref{eq:G2 system} is satisfied with $f=-\frac{1}{4}u$:
\begin{eqnarray*}
d(*\phi)=du\wedge *\phi.
\end{eqnarray*}
\end{prop}

We now study the terms that appear in the Bianchi identity.
\begin{lemma}
 \label{lem:torsionform}
 The torsion form of the $G_2$-structure $\phi$ is given by
\begin{eqnarray}
\label{eq:torsion form}
H= t^2\left( \beta_1\wedge \sigma_1+\beta_2\wedge\sigma_2+\beta_3\wedge\sigma_3\right)-\frac{1}{2}e^ui_{\nabla^4u}\omega_1^2,
\end{eqnarray}
where $\nabla^4 u$ is the gradient of $u$ on $S$. 
\end{lemma}
\begin{proof}
Recall that $H=-*(d\phi-du\wedge\phi)$.
 We compute
\begin{eqnarray*}
d\phi-du\wedge \phi= t^3\left( \beta_1\wedge \sigma_{23}+\beta_2\wedge\sigma_{31}+\beta_3\wedge\sigma_{12}\right) -t^3du\wedge \sigma_{123}.
\end{eqnarray*}
 Then we have
\begin{eqnarray*} 
*\left( t^3(\beta_1\wedge \sigma_{23}+\beta_2\wedge\sigma_{31}+\beta_3\wedge\sigma_{12})\right) 
&=& -t^2\left( \beta_1\wedge \sigma_1+\beta_2\wedge\sigma_2+\beta_3\wedge\sigma_3\right),\\
*\left( du\wedge t^3\sigma_{123}\right) &=& (-1)^3i_{\nabla u}\left(* t^3\sigma_{123}\right),\\
&=& -\frac{1}{2}e^{2u}i_{\nabla^7u}\omega_1^2,\\
&=& \frac{-1}{2}e^ui_{\nabla^4u}\omega_1^2.
\end{eqnarray*}
Here $\nabla^7 u$ is the gradient of $\pi^*u$ on $M$, while $\nabla^4u$ is the gradient of $u$ on $S$. Note that $\nabla^7u$ is horizontal, and related to the gradient on $S$ by $\pi_*(\nabla^7\pi^*u)=e^{-u}\nabla^4u$. The result follows.
\end{proof}
\begin{lemma}
The following identities hold,
\begin{eqnarray}
\label{eq:dH}
 dH&=&t^2(\beta_1^2+\beta_2^2+\beta_3^2)-\frac{1}{2}e^udu\wedge(i_{\nabla^4u}\omega_1^2)-\frac{1}{2}e^ud\left( i_{\nabla^4u}\omega_1^2\right),\\
 \label{eq:scalarequation}
*_4dH&=& \Delta( e^u)-t^2\left(|\beta_1|^2+|\beta_2|^2+|\beta_3|^2\right),
\end{eqnarray}
where $\delta d=\Delta$ is the Laplace-Beltrami operator (with positive spectrum). 
\end{lemma}
Note that all of the forms on the right of \eqref{eq:dH} are $4$-forms on $S$, pulled back to $M$.
\begin{proof}
The result follows from the identities 
\begin{eqnarray*}
t^2\, d\bigg(\beta_1\wedge \sigma_1+\beta_2\wedge\sigma_2+\beta_3\wedge\sigma_3\bigg) &=&t^2(\beta_1^2+\beta_2^2+\beta_3^2) 
= -*_4t^2\,\bigg( |\beta_1|^2+|\beta_2|^2+|\beta_3|^2\bigg)
,\\
d\left(\frac{-1}{2} e^ui_{\nabla^4u}\omega_1^2\right)&=& \frac{-1}{2}e^udu\wedge(i_{\nabla^4u}\omega_1^2)-\frac{1}{2}e^ud\left( i_{\nabla^4u}\omega_1^2\right)=*_4\Delta(e^u).
\end{eqnarray*}
\end{proof}

We now introduce the instanton data. Changing tack slightly, we consider $A$ to be a connection on a vector bundle instead of principal bundle. 
 Let $(\mathcal{S},h)$ be a hyperk\"ahler manifold, with hyperk\"ahler triple $\{\omega_i\}$. Following Verbitsky \cite{Ve96}, a  Hermitian connection $A$ on the Hermitian vector bundle $\mathcal{E}$ is {\sl hyperholomorphic} if the curvature $F_A$ is of type $(1,1)$ with respect to the three complex structures $J_i$ associated to the K\"ahler forms $\omega_i$. 
From \cite{Ve96}, we have:
\begin{prop}
 \label{prop:hyperholomorphicBundles}
 On the hyperk\"ahler surface $(S,\om_1,\om_2,\om_3)$ (ie. for a $K3$ surface or an abelian surface), the following are equivalent for a complex vector bundle $E_S\to S$:
 \begin{enumerate}
  \item[i)] $E_S$ admits a hyperholomorphic connection.
  \item[ii)] For some $i=1,2,3$, $E_S$ is a polystable bundle of degree zero on $(S,\om_i)$.
  \item[iii)] For all $i=1,2,3$, $E_S$ is polystable of degree zero on $(S,\om_i)$.
 \end{enumerate}
\end{prop}
If these conditions are satisfied, the connection is Hermitian-Yang-Mills, with respect to each complex structure on $S$, and satisfies 
\begin{equation}
 \label{eq:instanton on S}
 F_A\wedge \om_i=0, \ \ \ \  i=1,2,3.
\end{equation}
We now return to the example at hand. Let $\theta_S$ be a hyperholomorphic connection on $E_S\to S$. Consider the bundle $E=\pi^*E_S$ on $M$ and pulled-back connection $\theta=\pi^*\theta_S$ on $E$. The curvature satisfies $F_\theta=\pi^*F_{\theta_S}$ and is hence a $\phi_{u,t}$-instanton, for any $u\in C^\infty(M)$ and $t>0$. We now return to the Bianchi identity. Under the above assumptions, the Bianchi identity becomes an equation on $S$, equivalent to 
\begin{eqnarray}
\label{eq:Bianchi identity scalar equation}
\Delta h -t^2\left(|\beta_1|^2+|\beta_2|^2+|\beta_3|^2\right) = *_4\langle F_\theta\wedge F_\theta\rangle,
\end{eqnarray}
where we set $h=e^u\in C^\infty(S)$. This scalar equation admits a solution if and only if the integrals over $S$ of the left and right hand sides are equal. This provides a topological obstruction that constrains the choice of $E$, in relation to the topology of the torus bundle $M\to S$.


\subsection{Families of examples over $K3$ surfaces}
Following \cite{GF19b}, we will now give more explicit descriptions of hyperholomorphic bundles $E$ such that Equation \eqref{eq:Bianchi identity scalar equation}, and thus \eqref{eq:G2 system}, admits a solution. We assume from now that $S$ is a $K3$ surface. For more physical relevance, we will consider a bundle $E_S$ of the form
$$
E_S=TS^{1,0}\oplus V,
$$
for $V$ a hyperholomorphic vector bundle of (complex) rank $r$ on $S$ with
$$
c_1(V)=c_1(TS^{1,0})=c_1(S)=0.
$$
We fix the pairing
\begin{equation}
 \label{eq:pairing}
\langle\,,\rangle=\frac{\alpha}{4}\left(  \tr_{\mathfrak{gl_r}} - \tr_{\mathfrak{gl_2}} \right)
\end{equation}
for some real constant $\alpha\in\R^*$ and where $\tr_{\mathfrak{gl_j}}$ stands for an invariant Hermitian product on $\mathfrak{gl_j}$ that extends the Killing form on $\mathfrak{sl_j}$. 
The connections of interest will be product connections $\theta_S=\nabla\times A$, and we will denote the induced quadratic curvature  expression by
$$
\langle F_{\theta_S}\wedge F_{\theta_S}\rangle =\frac{\alpha}{4}(\tr F_A\wedge F_A - \tr F_\nabla\wedge F_\nabla).
$$
On cohomology, since $c_1(S)=c_1(V)=0$, this gives $\langle F_\theta\wedge F_\theta\rangle =2\pi^2\alpha(c_2(V)-c_2(TS^{1,0}))$. 
We also denote the intersection form on second cohomology by
$$
 Q:  H^2(S,\Z) \times H^2(S,\Z)  \to  \Z,
$$
 so that
$$
Q([(2\pi)^{-1}\beta_j]):=Q([(2\pi)^{-1}\beta_j],[(2\pi)^{-1}\beta_j])=-\frac{1}{4\pi^2}\int_S \vert \beta_j \vert^2 \; \mathrm{dvol_S}.
$$
Combining these formulas, Equation \eqref{eq:Bianchi identity scalar equation} admits a solution if and only if 
\begin{eqnarray*}
t^2\sum_jQ([(2\pi)^{-1}\beta_j])=\frac{\alpha}{2}(c_2(V)-c_2(S)).
\end{eqnarray*}
Recall from Proposition \ref{prop:hyperholomorphicBundles} that if a complex vector bundle on $S$ has zero first Chern class and is stable with respect to a fixed K\"ahler structure $\om$ on $S$, then it is hyperholomorphic. The tangent bundle $TS^{1,0}$ is stable and satisfies $c_1(TS^{1,0})=0$ and $c_2(S)=c_2(TS^{1,0})=24$ (see \cite{BHPV}). To obtain the required vector bundle $E_S$, it is thus enough to find a stable vector bundle $V$ on $S$ such that $c_1(V)=0$, and
\begin{equation}
 \label{eq:constraintc2}
c_2(V)=c_2(S)+\frac{2t^2}{\alpha}\sum_{j=1}^3 Q\left(\left[\frac{1}{2\pi}\beta_j\right]\right). 
\end{equation}
Criteria for the existence of stable vector bundles satisfying this condition are given in an application of the results of Perego and Toma \cite{PT17} by the second author \cite[Lemma 2.3]{GF19b}. This gives the following result.
\begin{prop}
 \label{prop:stablebundlesK3}
 Let $\alpha\in\mathbb{R}^*$ and $r\in\mathbb{N}^*$ such that
 \begin{equation}
  \label{eq:integralratios}
 \frac{2t^2}{\alpha}\sum_{j=1}^3 Q\left(\left[\frac{1}{2\pi}\beta_j\right]\right)\in\mathbb{Z}
 \end{equation}
 and 
 \begin{equation}
 \label{eq:bound on rank}
 r\leq 24 + \frac{2t^2}{\alpha}\sum_{j=1}^3 Q\left(\left[\frac{1}{2\pi}\beta_j\right]\right).
 \end{equation}
 Then there exists a stable rank $r$ bundle $V$ on $(S,\om)$ with $c_1(V)=0$ and $c_2(V)$ satisfying \eqref{eq:constraintc2}.
\end{prop}

\begin{remark}\label{rem:infinite}
The intersection form on a $K3$ is even, so $Q([\frac{\beta_j}{2\pi}])\in -2\mathbb{N}$. Thus, if $\alpha<0$ is chosen so that \eqref{eq:integralratios} holds, we obtain solutions on bundles of any rank satisfying \eqref{eq:bound on rank}, while, by taking $t^2\in\frac{\alpha}{4}\N$ sufficiently large we obtain solutions of any rank. There is a restricted range of ranks of holomorphic vector bundles that give rise to solutions with $\alpha>0$. In particular, solutions exist for infinitely many different choices of $\beta_i$ and, for different values of $\alpha>0$, infinitely many different ranks and values of $c_2(V)$.
\end{remark}

From this result, we obtain examples of complex vector bundles $E_S$ on $K3$ surfaces with product hyperholomorphic connections such that the scalar equation \eqref{eq:Bianchi identity scalar equation} admits a solution. These connections pull back to $G_2$-instantons $\theta$ on $E=\pi^*E_S$, and the system \eqref{eq:G2 system} admits solutions of the form $(\phi_{t,u},-\frac{1}{4}u,\theta)$ on $M$. Thus, we have proven:
\begin{theorem}
 \label{theo:Ansatz}
 Let $S$ be a $K3$-surface, and let $\beta_1,\beta_2$ and $\beta_3$ be closed anti-self-dual $2$-forms such that $[\frac{1}{2\pi}\beta_j]\in H^2(S,\Z)$. Let $\pi: M \to S$ be the associated $T^3$-bundle. Let $\alpha,t$ satisfying \eqref{eq:integralratios} and $r\in\mathbb{N}^*$ satisfying \eqref{eq:bound on rank}. Let $V$ be a stable vector bundle of rank $r$ on $S$ with $c_1(V)=0$ and $c_2(V)$ as in \eqref{eq:constraintc2}.
 Then, there exists a smooth function $u$ on $S$ and a product hyperholomorphic connection $\theta_S$ on $TS^{1,0}\times V$ such that $(\phi_{u,t}, \pi^*\theta_S)$ solves the $G_2$-Strominger system \eqref{eq:G2 system}, with $\phi_{u,t}$ defined by \eqref{eq:phiut} and pairing $c$ as in \eqref{eq:pairing}.
\end{theorem}
The exact sequence $0\to \underline{\mathbb{R}^3}\to TM\to\pi^*TS\to 0$, together with the connection form $\sigma\in\Omega^1(M,\underline{\mathbb{R}^3})$ define a decomposition
\begin{eqnarray*}
TM\simeq \underline{\mathbb{R}^3}\oplus\pi^*TS.
\end{eqnarray*}
Since $\underline{\R^3}$ admits a flat connection $\nabla_T$, we can consider the product instanton $\nabla_T\times \theta$ on $TM\times \pi^*V$. As $\nabla_T$ is flat, the pair $(\phi_{t,u},\nabla_T\times\theta)$ still solves the system \eqref{eq:G2 system}. These solutions are more relevant to physics as the instantons are product connections with one component on the tangent bundle of $M$.

We can consider further the connection $\nabla$ on $TM$. This is given as a product connection $\nabla_T\times \nabla_S$ on $TM\simeq \underline{\mathbb{R}^3}\oplus \pi^*TS$ where $\nabla_T$ is the trivial flat connection on $\underline{\mathbb{R}^3}$ and $\nabla_S$ is the Levi-Civita connection for a hyperk\"ahler metric on $S$. In particular $\nabla$ is compatible with the product metric $g$ on $TM$. Let $ g_{u,t}$ be the metric on $M$ induced by the $G_2$-structure $\phi_{u,t}$. There exists an automorphism $h$ of the tangent bundle $TM$ such that $g_{u,t}=h^*g=g(h\cdot,h\cdot)$. Then, the gauge transformed connection $\nabla^h=h^{-1}\circ\nabla\circ h$ is compatible with the metric $g_{u,t}$, and since
\begin{eqnarray*}
F_{\nabla^h} &=& h^{-1}\circ F_\nabla\circ h,\\
\tr(F_{\nabla^h}\wedge F_{\nabla^h})&=& \tr(h^{-1}\circ F_\nabla\wedge F_\nabla\circ h) =\tr(F_\nabla\wedge F_\nabla),
\end{eqnarray*}
the configuration with $\nabla^h$ in place of $\nabla$ still satisfies the $G_2$-Strominger system, with metric connection on $TM$.

\begin{remark}\label{rem:infinitesol}
 As the connection $A$ used to provide the hyperholomorphic connection $\theta_S$ satisfies the Hermite-Einstein equation with respect to a hermitian metric on $V$, we could have considered the $G_2$-Strominger system with principal bundle $P_K$, the pull back of the bundle of unitary frames on $V$, instead of $V$, as in the introduction.
\end{remark}

\begin{remark}\label{rem:compare}
In the case $S=T^4$, our result in Theorem \ref{theo:Ansatz} shall be compared with the solutions built on nilmanifolds in \cite{FIUV7}. Observe that the ansatz here is genuinely different as, for instance, the connections $\nabla$ and $A$ in \cite{FIUV7} depend non-trivially on the torus fibres, while in our case are given by pull-back from the base. We should stress that, unlike in \cite{FIUV7}, for different values of the parameters $t$, $\alpha$ and $r$, in Theorem \ref{theo:Ansatz} we obtain an infinite family of solutions for infinitely many different instanton bundles (see Remark \ref{rem:infinite}).
\end{remark}

\begin{remark}
In principle, using an implicit function theorem, one should be able to show that the solutions built in Theorem \ref{theo:Ansatz} vary in continuous families. More precisely, if $(S_s,V_s)$ is a smooth family of deformations of $K3$ surfaces $S_s$ with stable holomorphic vector bundles $V_s$, together with ASD forms $\beta_{i,s}$ as in Theorem \ref{theo:Ansatz}, and if $\alpha,t$ satisfy \eqref{eq:integralratios}, then we expect that the associated fonctions $u_s$ and connections $\theta_{S_s}$ can be taken to vary differentiably with $s$. These families of deformations would be constrained by the conditions of preserving the line bundles associated to the forms $\beta_{i,0}$. Their dimensions would be bounded by the  sum of the dimensions of the spaces $H^{0,1}(S,TS^{1,0})$ and $H^{0,1}(S,\mathrm{End}(V))$. These assertions are at this point only heuristic, but they motivate the question of deformations and moduli of solutions to the system as studied in Section \ref{sec:InfMod}.
\end{remark}

\subsection{Coassociative submanifolds of $M$} 
In this section we study distinguished submanifolds for the $G_2$-geometry constructed in Theorem \ref{theo:Ansatz},
following Harvey and Lawson \cite{HL82}. Let $(M,\phi)$ be $7$-manifold equipped with $G_2$-structure. An oriented $3$-dimensional submanifold $X\subseteq M$ is said to be {\sl associative} if $\phi$ restricts to $X$ as the Riemannian volume form of the induced metric on $X$. An oriented $4$-dimensional submanifold $Y\subseteq M$ is {\sl coassociative} of $*\phi$ restricts to be the volume form on $Y$. If $d\phi=0$, resp. $d*\phi=0$, then any closed associative submanifold, resp. coassociative submanifold, is volume minimizing among all cycles in its homology class. In particular, they give minimal submanifolds determined by purely {\sl first-order} differential conditions.

As above, let $(S,h,\omega_i)$ be a $K3$ surface endowed with hyperk\"ahler metric and hyperk\"ahler triple. We suppose that $\beta_1,\beta_2$ and $\beta_3$ are closed anti-self-dual $2$-forms with $[\beta_i]\in 2\pi H^2(S,\mathbb{Z})$, and that $\pi_i:P_i\to S$ are the $S^1$-bundles over $S$ with $c_1(P_i)=[(1/2\pi)\beta_i]$. The $7$-manifold $M$ is the fibre product of the three $P_i$. An elementary observation from formula \eqref{eq:phiut} is that $M \to S$ is an associative fibration, that is, it is fibred by associative submanifolds for any of the  $G_2$-structures $\phi_{u,t}$.

To find other interesting submanifolds, let $L\subseteq S$ be a complex curve in $S$, holomorphic with respect to the complex structure $J_1$ associated to the K\"ahler form $\omega_1$, such that $\beta_1|_L\equiv 0$. Then, $\pi_1^{-1}(L)\subseteq P_1$ is a smooth submanifold and the horizontal distribution given by the $1$-form $\sigma_1$ is Frobenius-integrable. That is, through each $s\in\pi_1^{-1}(L)$ there is a maximal integral submanifold $L_s$ that projects by $\pi_1$ as a local homeomorphism onto $L$. Furthermore, $\sigma_1|_{L_s}\equiv0$. Let $X_s\subseteq M$ be the fibre product of $L_s,P_2|_L$ and $P_3|_L$ over $L$. Then we have the result analogous to that of Goldstein and Prokushkin \cite{GP04} in the $6$-dimensional case.
\begin{prop}
Let $L\subseteq S$ be a smooth $2$-dimensional submanifold that is holomorphic with respect to $J_1$. Suppose that $\beta_1|_L\equiv 0$. Then, for $s\in \pi_1^{-1}(L)$, the immersed submanifold $X_s\subseteq M$ is coassociative with respect to the $G_2$-structure $\phi_{u,t}$, for any $u\in C^\infty(S)$, $t>0$.
\end{prop}
In particular, $X_s$ is homologically volume minimizing when $M$ is endowed with the coclosed $G_2$-structure $\phi^\prime=e^{-3u/4}\phi_{u,t}$. The proof of this result is immediate from the results of \cite{HL82}. The coassociative condition is equivalent to $\phi|_{X_s}=0$, which follows since $\omega_2|_L=\omega_3|_L=0$ and $\sigma_1|_{X_s}=0$. 

More can be said if $L\subseteq S$ is a smooth rational curve, diffeomorphic to $S^2$. In this case,  the maximal integral submanifold $L_s$ projects diffeomorphically onto $L$, and $X_s$ is a closed submanifold diffeomorphic to an smooth elliptic surface over $\mathbb{P}^1$.

\section{$T$-dual solutions}
\label{sec:Tdual}
In this section, we show that examples of solutions of \eqref{eq:G2 system} built in Theorem \ref{theo:Ansatz}, for different $\beta_j$'s and $t$'s, are $T$-dual. We first recall the definitions relevant to $T$-duality, and then construct explicit pairs of $T$-dual solutions.
\subsection{Background on $T$-duality}
\label{sec:Tduality}
There are two different points of view on $T$-duality that will be used in the next section: a topological one back to the work of Bouwknegt, Evslin, and Mathai \cite{BEM}, and a more refined geometric point of view given as an isomorphism of Courant algebroids originally observed by Cavalcanti and Gualtieri \cite{CaGu}. The specific form of topological $T$-duality that we will need was introduced by Baraglia and Hekmati in \cite{BarHek}, and involves principal bundles (see Definition \ref{def:Tdual}). The geometric version of this $T$-duality will be used in the proof of Theorem \ref{theo:Tduals}.


Let $G$ be a compact semisimple Lie group endowed with a symmetric non-degenerate invariant bilinear form $\langle\,,\rangle \in S^2(\mathfrak{g}^*)$ on its Lie algebra $\mathfrak{g}$.
Let $\omega$ be the $\mathfrak{g}$-valued Maurer-Cartan one-form on $G$ and $ \sigma^3$ the corresponding biinvariant Cartan three-form:
$$
\sigma^3 = - \frac{1}{6}\langle \omega , [\omega,\omega]\rangle.
$$
Let $M$ be a smooth manifold and $p:P \to M$ be a smooth principal $G$-bundle over $M$. Recall from \cite[Proposition 2.16]{Redden}:
\begin{defn}
\label{def:stringclass}
The space $H^3_{str}(P,\R)$ of string classes on $P$ is the torsor over $H^3(M,\R)$ of classes $\tau \in H^3(P,\R)$ which restrict to $[\sigma^3] \in H^3(G,\R)$ on the fibres of $P$, where, for $[H] \in H^3(M,\R)$ and $\tau \in H^3_{str}(P,\R)$ the action is given by
$
\tau \to \tau + p^*[H].
$
\end{defn}
Note that string classes are $G$-invariant classes on $P$. Indeed, for a given connection $\theta$ on $P$, string classes admit representatives of the form
\begin{equation}
\label{eq:hatH}
\hat H = p^* H + CS(\theta),
\end{equation}
where $CS(\theta)$ denotes the Chern-Simons three-form
\begin{equation*}
\label{eq:CS}
CS(\theta) = -  \frac{1}{6}\langle\theta,[\theta,\theta]\rangle + \langle F_\theta \wedge \theta\rangle \in \Omega^3(P).
\end{equation*}
\begin{remark}
\label{rem:stringclasssolution}
By construction, the Chern-Simons $3$-form satisfies
\begin{equation}
 \label{eq:chernsimonsFF}
dCS(\theta) = \langle F_\theta \wedge F_\theta\rangle.
\end{equation}
Thus, for a given string class represented by $\hat H=p^*H+CS(\theta)$ as in \eqref{eq:hatH}, the quantity $dH+\langle F_\theta\wedge F_\theta\rangle $ vanishes on $M$. In the other direction, assuming $M$ to be $7$-dimensional, to any solution $(\phi,\theta)$ of the $G_2$-Strominger system \eqref{eq:G2 system} on $M$, one can assign a string class
$$
\tau_{\phi,\theta}:=[-p^*H+CS(\theta)]\in H^3_{str}(P,\R),
$$ 
with $H$ the torsion form of $\phi$, as in equation \eqref{eq:torsion form}.
\end{remark}
Assume now, as in Section \ref{sec:ansatzconstruction}, that $M$ is itself the total space of a principal torus bundle over a base manifold $S$, with fibre a $k$-dimensional torus $T^k$, and $P$ is the pull-back of a principal $G$-bundle $P_s$ over $S$. Then, $P$ has a natural structure of $T^k \times G$-principal bundle, and we will consider $T^k \times G$-invariant string classes on $P$.
Then we can define $T$-duality, following \cite{BarHek}.
\begin{defn}
\label{def:Tdual}
Let $(M,P,\tau)$ and $(M',P',\tau') $ be triples where $(M,P)$ and $(M',P')$ are $G$-bundles pulled-back from a $G$-bundle $P_s\to S$, and where $\tau$ (resp. $\tau'$) is a $T^k \times G$-invariant string class on $P$ (resp. on $P'$). Then $(M,P,\tau)$ is \emph{$T$-dual} to $(M',P',\tau')$ if there exists a commutative diagram
\begin{equation}
 \label{eq:Mariosfunnydiagramm}
  \xymatrix{
 & \ar[ld]^{q} P \times_{P_s} P' \ar[d] \ar[rd]_{q'} & \\
 P \ar[d]_p \ar[r] &  P_s \ar[d] & \ar[l] \ar[d]^{p'} P' \\
 M \ar[r]_\pi & S & M' \ar[l]^{\pi'} \\
  }
\end{equation}
and representatives $\hat H$ and $\hat H'$ of the form \eqref{eq:hatH} of the string classes $\tau$ and $\tau'$, respectively, such that
\begin{equation}
\label{eq:dF}
dF = q^* \hat H - q'^*\hat H',
\end{equation}
for $F \in \Omega^2(P \times_{P_s} P')$ a $T^k \times T^{k'}$-invariant two-form on $P \times_{P_s} P'$ inducing a non-degenerate pairing
$$
F \colon \operatorname{Ker} dq \otimes \operatorname{Ker} dq' \to \R.
$$
\end{defn}
The relevance of $T$-duality in our construction comes from the fact that any solution built in Theorem \ref{theo:Ansatz} provides a triple $(M,P,\tau)$ as in Definition \ref{def:Tdual}. Furthermore, by \cite{GF19a}, if $(M',P',\tau')$ is $T$-dual to $(M,P,\tau)$, it also admits a solution to the $G_2$-Strominger system \eqref{eq:G2 system}.
 Indeed, as explained in \cite{grt}, a solution to the system \eqref{eq:G2 system} is equivalent to a solution of the {\sl Killing spinor equations} on a specific Courant algebroid associated to $(M,P)$ (the results from \cite{grt} are stated for the Hull-Strominger system in dimension $6$, but it is not difficult to see that they extend to the higher dimensional analogues, using \cite[Lemma 5.1]{grt} and  \cite[Theorem 1.2]{FrIv03}). In this language, $T$-duality becomes an isomorphism of Courant algebroids \cite{CaGu}, and solutions to the Killing spinors equation are transported through this isomorphism \cite{GF19a}. In the next section we provide explicit examples of this duality.

\subsection{Examples of $T$-dual solutions}
\label{sec:examplesTduals}

Let $S$ be a $K3$-surface with hyperk\"ahler triple $(\om_1,\om_2,\om_3)$ as in Section \ref{sec:ansatzconstruction}. Let $\beta=(\beta_1,\beta_2,\beta_3)$ be a triple of closed anti-selfdual $2$-forms such that $[\frac{1}{2\pi}\beta_j]\in H^2(S,\Z)$. Let $\alpha,t$ satisfying \eqref{eq:integralratios} and $r\in\N^*$ satisfying \eqref{eq:bound on rank}. Let $V$ be a smooth hermitian vector bundle of rank $r$ on $S$ with $c_1(V)=0$ and $c_2(V)$ as in \eqref{eq:constraintc2}, and fix a hermitian metric on $TS^{0,1}$. Let $G=U(2)\times U(r)$ and let $P_s$ be the $G$-principal bundle of split hermitian frames on $TS^{0,1}\oplus V$. Fix the bilinear pairing on $\mathfrak{g}$, the Lie algebra of $G$, to be the restriction of the pairing considered in \eqref{eq:pairing}:
$$
\langle\,,\rangle=\frac{\alpha}{4} (\mathrm{tr}_{\mathfrak{u}(r)}-\mathrm{tr}_{\mathfrak{u}(2)}).
$$

We assume from now that $t$ satisfies the additional constraints, for $j\in\lbrace 1,2,3\rbrace$,
\begin{equation}
 \label{eq:tbetaintegral}
\left[\frac{t^2}{2\pi}\beta_j\right]\in H^2(S,\Z).
\end{equation}
Set now 
$$
t' =t^{-1}
$$
and 
$$
\beta'=-t^2 \beta.
$$
We can consider the $T^3$-bundles associated to $\beta$:
$$
\pi_\beta: M \to S,
$$
and to $\beta'$:
$$
\pi_{\beta'}: M' \to S.
$$
Set $P=\pi_\beta^*P_s$ and $P'=\pi_{\beta'}^*P_s$  pulled back $G$-bundles on $M$ and $M'$ respectively. Denoting $q$ (resp. $q'$) the projection map from $P\times_{P_s} P'$ to $P$ (resp. to $P'$), we are in the situation of diagram \eqref{eq:Mariosfunnydiagramm}.

Then, replacing $(t,\beta)$ by $(t',\beta')$, the integrality condition \eqref{eq:integralratios} is preserved while the quantity on the right hand side of \eqref{eq:constraintc2} is fixed.
Thus, for both sets of data $(\alpha,t,\beta)$ and $(\alpha,t',\beta')$ we are in the situation of Theorem \ref{theo:Ansatz}, and we can find solutions $(\phi_{u,t},\pi_{\beta}^*\theta_s)$ and $(\phi_{u',t'},\pi_{\beta'}^*\theta_s)$ of the $G_2$-Strominger system on $M$ and $M'$ respectively. Note that $u$ and $u'$ actually solve the same equation \eqref{eq:Bianchi identity scalar equation} so we will assume $u=u'$. Let $\tau$ (resp. $\tau'$) be the string class of $(\phi_{u,t},\pi_\beta^*\theta_s)$ (resp. of $(\phi_{u,t'},\pi_{\beta'}^*\theta_s)$) as in Remark \ref{rem:stringclasssolution}. Then these two sets of solutions are actually $T$-dual.
\begin{theorem}
 \label{theo:Tduals}
 Under the above asumptions, $(M,P,\tau)$ is $T$-dual to $(M',P',\tau')$. Moreover, the solutions $(\phi_{u,t},\pi_\beta^*\theta_s)$ and $(\phi_{u,t'},\pi_{\beta'}^*\theta_s)$ of the $G_2$-Strominger system \eqref{eq:G2 system} on $M$ and $M'$ are exchanged under this $T$-duality.
\end{theorem}
\begin{proof}
 Denote as in Section \ref{sec:ansatzconstruction} by $\sigma$ (resp. $\sigma'$) the $\R^3$-valued connection $1$-form of $P$ (resp. of $P'$). Then, using Lemma \ref{lem:torsionform},  we can compute a representative $\hat H$ (resp. $\hat H'$) for $\tau$ (resp. for $\tau'$) as in \eqref{eq:hatH}:
 $$
 \hat H= H  + CS(\theta_s), \qquad \hat H'= H'  + CS(\theta_s)
 $$
where 
$$
H = -t^2 \sum_{j=1}^3 \beta_j\wedge \sigma_j - \iota_{\nabla e^u} \frac{\om_1^2}{2}, \qquad H' = \sum_{j=1}^3 \beta_j\wedge \sigma_j' - \iota_{\nabla e^u} \frac{\om_1^2}{2}
$$
and we omitted the pullbacks to ease notations.
 Then, as the diagram \eqref{eq:Mariosfunnydiagramm} commutes, we obtain
 $$
 q^*\hat H -q'^*\hat H'=-t^2\sum_{j=1}^3 \beta_j\wedge \sigma_j -\sum_{j=1}^3 \beta_j\wedge \sigma_j',
 $$
 and thus 
 $$
  q^*\hat H -q'^*\hat H'=-d \sum_{j=1}^3 \sigma_j\wedge \sigma_j'.
 $$
 As $\sum_{j=1}^3 \sigma_j\wedge\sigma_j'$ is non-degenerate on $\mathrm{Ker} dq \otimes \mathrm{Ker} dq'$, we obtain that $(M,P,\tau)$ and $(M',P',\tau')$ are $T$-dual. 
 
Arguing now as in the proof of \cite[Theorem 3.5]{GF19b}, it is not difficult to see that the triples 
$(g_{\phi_{u,t}},H,\pi_\beta^*\theta_s)$ and $(g_{\phi_{u,t'}},H',\pi_{\beta'}^*\theta_s)$--where the metrics are defined as in \eqref{eq:metricG2}--are $T$-dual in the sense that they are exchanged by the isomorphism of Courant algebroids in \cite[Proposition 2.11]{BarHek}. This follows regarding these triples as defining generalized metrics on a transitive Courant algebroid \cite[Proposition 3.4]{GF14} and applying \cite[Proposition 4.13]{BarHek} (see \cite[Definition 6.2]{GF19a} for a precise definition of $T$-dual metrics in the present context). To finish, note that the $G_2$-structures
of this pair of solutions are completely determined by the holonomy bundle of the canonical connection with skew symmetric torsion (e.g. for $(\phi_{u,t},\pi_\beta^*\theta_s)$) \cite{FrIv02}
$$
\nabla^{\phi_{u,t'}} + \frac{1}{2} H.
$$
This implies that $(\phi_{u,t},\pi_\beta^*\theta_s)$ is $T$-dual to $(\phi_{u,t'},\pi_{\beta'}^*\theta_s)$.
\end{proof}

\begin{remark}
 In the construction of Section \ref{sec:ansatzconstruction}, the parameter $t$ appears as a free parameter constraining $\alpha$. Thus it is easy to find pairs $t$ and $\beta$ that satisfies the additional integrality condition \eqref{eq:tbetaintegral} required in $T$-duality.
\end{remark}


\bibliographystyle{plain}	

\bibliography{bibCGFT}

\begin{thebibliography}{10}

\bibitem{Acharya}
B. S. Acharya.
\newblock On mirror symmetry for manifolds of exceptional holonomy.
\newblock {\em Nuclear Phys. B} 524(1-2):269--282, 1998.

\bibitem{AcharyaGukov}
B. S. Acharya and S. Gukov.
\newblock M theory and singularities of exceptional holonomy manifolds.
\newblock {\em Phys. Rep.} 392(3):121--189, 2004.

\bibitem{ASTW19}
Anthony Ashmore, Charles Strickland-Constable, and Daniel Waldram.
\newblock Generalising {$G_2$} geometry: involutivity, moment maps and moduli.
\newblock ArXiv preprint 1910.04795.

\bibitem{BaOl}
Gavin Ball and Goncalo Oliveira.
\newblock Gauge theory on {A}loff-{W}allach spaces.
\newblock {\em Geom. Topol.}, 23(2):685--743, 2019.

\bibitem{BarHek}
D.~Baraglia and P.~Hekmati.
\newblock Transitive courant algebroids, string structures and {T}-duality.
\newblock {\em Adv. Theor. Math. Phys.}, 19:613--672, 2015.

\bibitem{BHPV}
Wolf~P. Barth, Klaus Hulek, Chris A.~M. Peters, and Antonius Van~de Ven.
\newblock {\em Compact complex surfaces}, volume~4 of {\em Ergebnisse der
  Mathematik und ihrer Grenzgebiete. 3. Folge. A Series of Modern Surveys in
  Mathematics [Results in Mathematics and Related Areas. 3rd Series. A Series
  of Modern Surveys in Mathematics]}.
\newblock Springer-Verlag, Berlin, second edition, 2004.

\bibitem{BEM}
Peter Bouwknegt, Jarah Evslin, and Varghese Mathai.
\newblock T-duality: topology change from H-flux.
\newblock {\em Comm. Math. Phys.} 249: 383--415, 2004.

\bibitem{Bryant}
Robert~L. Bryant.
\newblock Some remarks on {$G_2$}-structures.
\newblock In {\em Proceedings of {G}\"{o}kova {G}eometry-{T}opology {C}onference 2005}, pages 75--109. G\"{o}kova Geometry/Topology Conference
  (GGT), G\"{o}kova, 2006.

\bibitem{CaGu}
Gil~R. Cavalcanti and Marco Gualtieri.
\newblock Generalized complex geometry and {T}-duality.
\newblock {\em Am. Math. Soc. (CRM Proceedings \& Lecture Notes)}, A Celebration of the Mathematical Legacy of Raoul Bott:341--366, 2010.

\bibitem{CMTW}
A. Coimbra, R. Minasian, H. Triendl, and D. Waldram,
\newblock Generalised geometry for string corrections,
\newblock JHEP, 2014.

\bibitem{CHNP}
Alessio Corti, Mark Haskins, Johannes Nordstr\"om, and Tommaso Pacini.
\newblock {$G_2$}-manifolds and associative submanifolds via semi-Fano 3-folds.
\newblock {\em Duke Math. J.} 164(10):1971--2092, 2015.

\bibitem{DLMS19}
Xenia de~la Ossa, Magdalena Larfors, Matthew Magill, and Eirik~E. Svanes.
\newblock Superpotential of three dimensional $\mathcal{N}=1$ heterotic supergravity.
\newblock ArXiv preprint 1904.01027v2.

\bibitem{OssaLaSv15}
Xenia de~la Ossa, Magdalena Larfors, and Eirik~E. Svanes.
\newblock Exploring {$SU(3)$} structure moduli spaces with
integrable {$G2$} structures.
\newblock {\em Adv. Theor. Math. Phys.}, 19(4): 837--903, 2015.

\bibitem{OssaLaSv16}
Xenia de~la Ossa, Magdalena Larfors, and Eirik~E. Svanes.
\newblock Infinitesimal moduli of {$G2$} holonomy manifolds with instanton bundles.
\newblock {\em J. High Ener. Phys.}, 16(16), 2016.

\bibitem{DLS18b}
Xenia de~la Ossa, Magdalena Larfors, and Eirik~E. Svanes.
\newblock The infinitesimal moduli space of heterotic {$G_2$} systems.
\newblock {\em Comm. Math. Phys.}, 360(2):727--775, 2018.

\bibitem{DLS18a}
Xenia de~la Ossa, Magdalena Larfors, and Eirik~E. Svanes.
\newblock Restrictions of heterotic {$G_2$} structures and instanton
  connections.
\newblock In {\em Geometry and physics. {V}ol. {II}}, pages 503--517. Oxford
  Univ. Press, Oxford, 2018.

\bibitem{DS}
S.K. Donaldson and Ed~Segal.
\newblock Gauge theory in higher dimensions, {II}.
\newblock In {\em Surveys in differential geometry. {V}olume {XVI}. {G}eometry
  of special holonomy and related topics}, volume~16 of {\em Surv. Differ.
  Geom.}, pages 1--41. Int. Press, Somerville, MA, 2011.

\bibitem{DT}
S.K. Donaldson and R.~P. Thomas.
\newblock Gauge theory in higher dimensions.
\newblock In {\em The geometric universe ({O}xford, 1996)}, pages 31--47.
  Oxford Univ. Press, Oxford, 1998.

\bibitem{DN}
Avron Douglis and Louis Nirenberg.
\newblock Interior estimates for elliptic systems of partial differential equations.
\newblock {\em Comm. Pure Appl. Math.}, 8:503--538, 1955.

\bibitem{Fei} 
T. Fei. 
\newblock Generalized Calabi-Gray Geometry and Heterotic Superstrings, 
\newblock arXiv:1807.08737.

\bibitem{FerGr}
M.~Fern\'{a}ndez and A.~Gray.
\newblock Riemannian manifolds with structure group {$G_{2}$}.
\newblock {\em Ann. Mat. Pura Appl. (4)}, 132:19--45 (1983), 1982.

\bibitem{FIUVa7}
Marisa Fern\'{a}ndez, Stefan Ivanov, Luis Ugarte, and Dimiter Vassilev.
\newblock Quaternionic {H}eisenberg group and heterotic string solutions with
  non-constant dilaton in dimensions 7 and 5.
\newblock {\em Comm. Math. Phys.}, 339(1):199--219, 2015.

\bibitem{FIUV7}
Marisa Fern\'{a}ndez, Stefan Ivanov, Luis Ugarte, and Raquel Villacampa.
\newblock Compact supersymmetric solutions of the heterotic equations of motion
  in dimensions 7 and 8.
\newblock {\em Adv. Theor. Math. Phys.}, 15(2):245--284, 2011.

\bibitem{FQS}
M.-A. Fiset, C. Quigley, E. Svanes
\newblock Marginal deformations of heterotic G2 sigma models
\newblock {\em Journal of High Energy Physics}, 2007.

\bibitem{FrUhl}
Daniel~S. Freed and Karen~K. Uhlenbeck.
\newblock {\em Instantons and four-manifolds}, volume~1 of {\em Mathematical
  Sciences Research Institute Publications}.
\newblock Springer-Verlag, New York, 1984.

\bibitem{FrIv02}
Thomas Friedrich and Stefan Ivanov.
\newblock Parallel spinors and connections with skewsymmetric torsion in string
  theory.
\newblock {\em Asian J. Math.}, 6:303--336, 2002.

\bibitem{FrIv03}
Thomas Friedrich and Stefan Ivanov.
\newblock Killing spinor equations in dimension 7 and geometry of integrable
  {$G_2$}-manifolds.
\newblock {\em J. Geom. Phys.}, 48(1):1--11, 2003.

\bibitem{FY08}
Ji-Xiang Fu and Shing-Tung Yau.
\newblock The theory of superstring with flux on non-{K}\"{a}hler manifolds and
  the complex {M}onge-{A}mp\`ere equation.
\newblock {\em J. Differential Geom.}, 78(3):369--428, 2008.

\bibitem{GF14}
Mario Garcia-Fernandez.
\newblock Torsion-free generalized connections and heterotic supergravity.
\newblock {\em Comm. Math. Phys.}, 332(1):89--115, 2014.

\bibitem{GF19a}
Mario Garcia-Fernandez.
\newblock Ricci flow, {K}illing spinors, and {T}-duality in generalized geometry.
\newblock {\em Adv. Math.}, 350:1059--1108, 2019.

\bibitem{GF2}
Mario Garcia-Fernandez.
\newblock Lectures on the Strominger system.
\newblock {\em Travaux Math\'ematiques, Special Issue: School GEOQUANT at the ICMAT}, Vol. XXIV, 7--61 2016.

\bibitem{GF19b}
Mario Garcia-Fernandez.
\newblock T-dual solutions of the {H}ull-{S}trominger system on non-{K}\"ahler
  threefolds.
\newblock {\em Crelle's Journal}, 2019.

\bibitem{grt}
Mario Garcia-Fernandez, Roberto Rubio, and Carl Tipler.
\newblock Infinitesimal moduli for the {S}trominger system and {K}illing spinors in generalized geometry.
\newblock {\em Math. Ann.}, 369(1-2):539--595, 2017.

\bibitem{GaMaWa}
Jerome~P. Gauntlett, Dario Martelli, and Daniel Waldram.
\newblock Superstrings with intrinsic torsion.
\newblock {\em Phys. Rev. D (3)}, 69(8):086002, 27, 2004.

\bibitem{GP04}
Edward Goldstein and Sergey Prokushkin.
\newblock Geometric model for complex non-{K}\"{a}hler manifolds with {${\rm
  SU}(3)$} structure.
\newblock {\em Comm. Math. Phys.}, 251(1):65--78, 2004.

\bibitem{G3}
Marco Gualtieri.
\newblock Branes on {P}oisson varieties.
\newblock In {\em The many facets of geometry}, pages 368--394. Oxford Univ.
  Press, Oxford, 2010.

\bibitem{GuNi}
Murat G\"{u}naydin and Hermann Nicolai.
\newblock Seven-dimensional octonionic {Y}ang-{M}ills instanton and its
  extension to an heterotic string soliton.
\newblock {\em Phys. Lett. B}, 351(1-3):169--172, 1995.

\bibitem{HMS}
Nick Halmagyi, Ilarion V. Melnikov, Savdeep Sethi,
\newblock Instantons, Hypermultiplets and the Heterotic String,
\newblock {\em JHEP 07}, 07: 086, 2007.

\bibitem{HN}
Derek Harland and Christoph N\"{o}lle.
\newblock Instantons and {K}illing spinors.
\newblock {\em J. High Energy Phys.}, (3):082, front matter+37, 2012.

\bibitem{HL82}
Reese Harvey and H.~Blaine Lawson, Jr.
\newblock Calibrated geometries.
\newblock {\em Acta Math.}, 148:47--157, 1982.

\bibitem{Ha12}
Andriy Haydys.
\newblock Gauge theory, calibrated geometry and harmonic spinors.
\newblock {\em J. Lond. Math. Soc. (2)}, 86(2):482--498, 2012.

\bibitem{Hit0}
Nigel Hitchin.
\newblock The geometry of three-forms in six and seven dimensions.
\newblock ArXiv preprint 0010054.

\bibitem{Hit1}
Nigel Hitchin.
\newblock Generalized {C}alabi-{Y}au manifolds.
\newblock {\em Q. J. Math.}, 54(3):281--308, 2003.

\bibitem{HullTurin} 
C. Hull. 
\newblock Superstring compactifications with torsion and space-time supersymmetry.
\newblock {\em In Turin 1985 Proceedings ``Superunification and Extra Dimensions''}: 347--375, 1986.

\bibitem{Ivan09}
Stefan Ivanov.
\newblock Heterotic supersymmetry, anomaly cancellation and equations of
  motion.
\newblock {\em Phys. Lett. B}, 685(2-3):190--196, 2010.

\bibitem{joy}
Dominic~D. Joyce.
\newblock {\em Compact manifolds with special holonomy}.
\newblock Oxford Mathematical Monographs. Oxford University Press, Oxford, 2000.

\bibitem{Leung}
J.-H. Lee and N.-C. Leung.
\newblock Geometric structures on {$G_2$} and {$Spin(7)$}-manifolds.
\newblock {\em Adv. Theor. Math. Phys.}, 13(1):1--31, 2009.


\bibitem{LoOl}
Jason~D. Lotay and Goncalo Oliveira.
\newblock {$\rm SU(2)^2$}-invariant {$G_2$}-instantons.
\newblock {\em Math. Ann.}, 371(1-2):961--1011, 2018.

\bibitem{MelSS}
I.-V. Melnikov, S. Sethi, and E. Sharpe.
\newblock Recent Developments in (0,2) Mirror Symmetry
\newblock {\em SIGMA}, 8:068, 2018.

\bibitem{O}
Gon\c{c}alo Oliveira.
\newblock Monopoles on the {B}ryant-{S}alamon {$G_2$}-manifolds.
\newblock {\em J. Geom. Phys.}, 86:599--632, 2014.

\bibitem{PT17}
Arvid Perego and Matei Toma.
\newblock Moduli spaces of bundles over nonprojective {K}3 surfaces.
\newblock {\em Kyoto J. Math.}, 57(1):107--146, 2017.

\bibitem{PPZ2} 
D.-H. Phong, S. Picard, and X. Zhang, 
New curvature flows in complex geometry. 
\newblock {\em Surveys in Differential Geometry}, 22:331--364, 2017.

\bibitem{Redden}
C.~Redden.
\newblock String structures and canonical $3$-forms.
\newblock {\em Pac. J. Math.}, 249:447--484, 2011.

\bibitem{SEthesis}
Henrique~N. S\'{a}~Earp.
\newblock {\em Instantons on $G_2$-manifolds}.
\newblock 2009.
\newblock Thesis (Ph.D.)--Imperial College London.

\bibitem{SW}
Henrique~N. S\'{a}~Earp and Thomas Walpuski.
\newblock {$\rm {G}_2$}-instantons over twisted connected sums.
\newblock {\em Geom. Topol.}, 19(3):1263--1285, 2015.

\bibitem{Strom}
Andrew Strominger.
\newblock Superstrings with torsion.
\newblock {\em Nuclear Phys. B}, 274(2):253--284, 1986.

\bibitem{Ta12}
Yuuji Tanaka.
\newblock A construction of {$Spin(7)$}-instantons.
\newblock {\em Ann. Global Anal. Geom.}, 42(4):495--521, 2012.

\bibitem{Ti00}
Gang Tian.
\newblock Gauge theory and calibrated geometry. {I}.
\newblock {\em Ann. of Math. (2)}, 151(1):193--268, 2000.

\bibitem{TsengYau}
Li-Sheng Tseng and Shing-Tung Yau.
\newblock Non-{K}\"{a}hler {C}alabi-{Y}au manifolds.
\newblock In {\em String-{M}ath 2011}, volume~85 of {\em Proc. Sympos. Pure
  Math.}, pages 241--254. Amer. Math. Soc., Providence, RI, 2012.

\bibitem{Ve96}
Mikhail Verbitsky.
\newblock Hyperholomorphic bundles over a hyper-{K}\"{a}hler manifold.
\newblock {\em J. Algebraic Geom.}, 5(4):633--669, 1996.

\bibitem{Wa}
Thomas Walpuski.
\newblock {$\rm G_2$}-instantons on generalised {K}ummer constructions.
\newblock {\em Geom. Topol.}, 17(4):2345--2388, 2013.

\end{thebibliography}

\end{document}